\documentclass[english]{article}
\usepackage[T1]{fontenc}
\usepackage[latin9]{inputenc}
\usepackage{geometry}
\usepackage{color}
\usepackage{array}
\usepackage{longtable}
\usepackage{float}
\usepackage{booktabs}
\usepackage{multirow}
\usepackage{amsmath}
\usepackage{amsthm}
\usepackage{amssymb}
\usepackage{graphicx}

\makeatletter

\newcommand{\noun}[1]{\textsc{#1}}
\providecommand{\tabularnewline}{\\}
\floatstyle{ruled}
\newfloat{algorithm}{tbp}{loa}
\providecommand{\algorithmname}{Algorithm}
\floatname{algorithm}{\protect\algorithmname}

\theoremstyle{plain}
\newtheorem{thm}{\protect\theoremname}
\theoremstyle{plain}
\newtheorem{lem}[thm]{\protect\lemmaname}
\ifx\proof\undefined
\newenvironment{proof}[1][\protect\proofname]{\par
	\normalfont\topsep6\p@\@plus6\p@\relax
	\trivlist
	\itemindent\parindent
	\item[\hskip\labelsep\scshape #1]\ignorespaces
}{%
	\endtrivlist\@endpefalse
}
\providecommand{\proofname}{Proof}
\fi
\theoremstyle{plain}
\newtheorem{prop}[thm]{\protect\propositionname}

\usepackage{algorithmic}
\def\vec#1{\boldsymbol{#1}}

\usepackage{algorithm, algorithmic}

\usepackage{framed}
\usepackage{listings}

\renewcommand\[{\begin{equation}}
\renewcommand\]{\end{equation}}

\makeatother

\usepackage{babel}
\providecommand{\lemmaname}{Lemma}
\providecommand{\propositionname}{Proposition}
\providecommand{\theoremname}{Theorem}

\begin{document}

\title{Robust Optimal-Complexity Multilevel ILU for \\
 Predominantly Symmetric Systems}

\author{Aditi Ghai\footnotemark[1] \and Xiangmin Jiao\footnotemark[1]\ \footnotemark[2]}

\date{}

\maketitle
\footnotetext[1]{Dept. of Applied Math. \& Stat., Stony Brook University, Stony Brook, NY 11794, USA.}
\footnotetext[2]{Corresponding author. Email: xiangmin.jiao@stonybrook.edu.}
\begin{abstract}
Incomplete factorization is a powerful preconditioner for Krylov subspace
methods for solving large-scale sparse linear systems. Existing incomplete
factorization techniques, including incomplete Cholesky and incomplete
LU factorizations, are typically designed for symmetric or nonsymmetric
matrices. For some numerical discretizations of partial differential
equations, the linear systems are often nonsymmetric but predominantly
symmetric, in that they have a large symmetric block. In this work,
we propose a multilevel incomplete LU factorization technique, called
\emph{PS-MILU}, which can take advantage of predominant symmetry to
reduce the factorization time by up to half. PS-MILU delivers robustness
for ill-conditioned linear systems by utilizing diagonal pivoting
and deferred factorization. We take special care in its data structures
and its updating and pivoting steps to ensure optimal time complexity
in input size under some reasonable assumptions. We present numerical
results with PS-MILU as a preconditioner for GMRES for a collection
of predominantly symmetric linear systems from numerical PDEs with
unstructured and structured meshes in 2D and 3D, and show that PS-MILU
can speed up factorization by about a factor of 1.6 for most systems.
In addition, we compare PS-MILU against the multilevel ILU in ILUPACK
and the supernodal ILU in SuperLU to demonstrate its robustness and
lower time complexity
\end{abstract}

\section{Introduction}

\textcolor{black}{Preconditioned Krylov subspace (KSP) methods are
widely used for solving sparse linear systems, especially those arising
from numerical discretizations of partial differential equations (PDEs).
These methods typically require some robust and efficient preconditioners
to be effective, especially for large-scale problems. Incomplete factorization
techniques, including incomplete LU factorization with pivoting for
nonsymmetric systems or incomplete Cholesky factorization for symmetric
and positive definite (SPD) systems, are among the most robust preconditioners,
and some of their variants are often quite efficient for linear systems
arising from PDE discretizations. In practice, some linear systems
are often nonsymmetric but have a symmetric block, and it is worth
exploring this partial symmetry to improve the robustness and efficiency
of the incomplete factorizations. Without loss of generality, we assume
the matrix is real and the symmetric part is the leading block; i.e.,
the matrix $\vec{A}\in\mathbb{R}^{n\times n}$ has the form}
\begin{equation}
\vec{A}=\begin{bmatrix}\vec{B} & \vec{F}\\
\vec{E} & \vec{C}
\end{bmatrix},\label{eq:predominant-matrix-form}
\end{equation}
where $\vec{B}\in\mathbb{R}^{n_{1}\times n_{1}}$ symmetric and $\vec{E}\neq\vec{F}^{T}$.
Note that form (\ref{eq:predominant-matrix-form}) includes symmetric
and nonsymmetric matrices as special cases, for which $\vec{C}$ and
$\vec{B}$ are empty, respectively. We are particularly interested
in the case where the size of $\vec{B}$ dominates that of $\vec{C}$,
and we refer to such systems as\textcolor{black}{{} }\textcolor{black}{\emph{predominantly
symmetric}}\textcolor{black}{. }These systems may arise from PDE discretizations.
For example, for the Poisson equation with \textcolor{black}{Dirichlet
boundary conditions, if the Dirichlet nodes are not eliminated from
the system, then we have $\vec{C}=\vec{I}$, $\vec{E}=\vec{0}$, and
$\vec{F}\neq\vec{0}$. Another example is the finite difference methods
for the Poisson equation with Neumann boundary conditions on a structured
mesh: If centered difference is used in the interior and one-sided
difference is used for Neumann boundary conditions, then we obtain
a predominantly symmetric system, where the rows in $\vec{B}$ correspond
to the interior nodes and those in $\vec{C}$ correspond to the Neumann
nodes. Similarly, for some variants of finite difference methods,
such as embedded boundary and immersed boundary methods for parabolic
or elliptic problems, the matrix block corresponding to the interior
nodes may be symmetric but that corresponding to near-boundary nodes
is in general nonsymmetric. Other examples include finite element
methods with a high-order treatment of Neumann boundary conditions
over curved domains, which modify the rows in the stiffness matrix
corresponding to Neumann nodes instead of simply substituting the
boundary conditions to the right-hand side vector; see e.g. \cite{bochev2017optimally}.
In all these examples,} due to the surface-to-volume ratio, the size
of $\vec{C}$ is much smaller than that of $\vec{B}$.

Given a predominantly symmetric matrix in form (\ref{eq:predominant-matrix-form}),
it is conceivable that one would like to take advantage of the symmetry
in $\vec{B}$ to reduce the factorization time, especially if the
size of $\vec{B}$ dominates that of $\vec{C}$. To the best of our
knowledge, there was no incomplete factorization technique in the
literature that can take advantage of this predominant symmetry. The
lack of such a method is probably because most incomplete factorization
techniques are based on variants of LU factorization with or without
column pivoting for nonsymmetric systems or Cholesky factorization
for SPD systems. Since LU with column pivoting would destroy the symmetry,
whereas LU without pivoting is unstable, there is no known stable
LU factorization techniques for predominantly symmetric systems. As
a result, it is a challenging task to develop an incomplete ILU algorithm
for predominantly symmetric systems, and it requires a drastically
different approach than traditional factorization techniques. Note
that the form (\ref{eq:predominant-matrix-form}) shares some similarity
with KKT systems\textcolor{black}{, which often require special solvers;
see \cite{BenziGolubLiesen05NSS} for an extensive survey. However,}
(\ref{eq:predominant-matrix-form}) is different from KKT systems
in that $\vec{E}\neq\vec{F}^{T}$ and $\vec{C}$ is typically nonzero
in (\ref{eq:predominant-matrix-form}). Although a preconditioner
for (\ref{eq:predominant-matrix-form}) may be applied to KKT systems,
the converse is in general not true.

In this paper, we propose an incomplete ILU technique, referred to
as \emph{PS-MILU}, for predominantly symmetric systems, which will
provide a robust, efficient, and unified factorization algorithm for
symmetric and nonsymmetric systems. We develop this method based on
the following key ideas. First, we introduce a modified version of
the Crout update procedure in \cite{li2003crout} to support symmetric
LDL$^{\text{T}}$ and nonsymmetric LDU factorizations with diagonal
pivoting. Second, to improve robustness, we adopt the framework of
multilevel ILU factorization \cite{Boll06MPC} to defer the factorizations
of rows and columns that would cause the norms of $\vec{L}^{-1}$
and $\vec{U}^{-1}$ to grow rapidly. Third, to achieve efficiency,
we limit the number of nonzeros in the approximate factors to be within
a constant factor of those in the input, and develop data structures
to ensure optimal time complexity of the overall algorithm. In particular,
we ensure that the cost of updating the nonzeros in the $\vec{L}$
and $\vec{U}$ factors is linear, assuming the number of nonzeros
per row and per column is bounded by a constant, and its cost dominates
all the other steps, including pivoting, sorting, dropping, and computing
of the Schur complement, etc. As a result, PS-MILU delivers a robust
and optimal-complexity method for most linear systems from PDE discretizations.
Furthermore, it can speed up the factorization by up to a factor of
two for predominantly symmetric systems. We present numerical results
with PS-MILU as a preconditioner for GMRES for a collection of linear
systems from numerical PDEs with unstructured and structured meshes
in 2D and 3D. In addition, we compare PS-MILU against the multilevel
ILU with diagonal pivoting in ILUPACK \cite{ilupack} and the supernodal
ILU with column pivoting in SuperLU \cite{lishao10}. Our numerical
results show that PS-MILU scales better than both ILUPACK and SuperLU,
while delivering comparable robustness.

\textcolor{black}{The remainder of the paper is organized as follows.
In Section~\ref{sec:background}, we review some background knowledge
on incomplete ILU factorization and its variants. In Section~\ref{sec:PS-MILU},
we describe the components of the proposed multilevel ILU for predominantly
symmetric systems. In Section\ \ref{sec:Implementation-details},
we present some implementation details of the algorithm,} with a focus
on the data structure, updating and pivoting, and the complexity analysis.\textcolor{black}{{}
In Section~\ref{sec:Numerical-Results}, we present some numerical
results with PS-MILU as a preconditioner for GMRES and compare its
performance with other techniques. Finally, Section~\ref{sec:Conclusions and Future Work}
concludes the paper with a discussion of future work.}

\section{Background\label{sec:background}}

The technique proposed in this work is based on multilevel ILU factorization,
which is one of the most effective preconditioner for Krylov subspace
methods. Multilevel ILU is a sophisticated variant of incomplete LU
factorization. In \textcolor{black}{\cite{ghai2017comparison}}, we
compared multilevel ILU against some other preconditioners (including
SOR, ILU with and without thresholding or pivoting \textcolor{black}{\cite{saad1994ilut}},
BoomerAMG \cite{hypre-user} and smoothed-aggregation AMG \cite{GeeSie06ML})
for several Krylov subspace methods (including GMRES \cite{Saad86GMRES},
BiCGSTAB \cite{vanderVorst92BiCGSTAB}, TFQMR \cite{Freund93TFQMR},
and QMRCGSTAB \cite{CGS94QMRCGS}). It was shown that although multigrid
methods are often the most efficient preconditioners when applicable,
thanks to its nearly linear scaling, incomplete LU techniques are
more robust for ill-conditioned systems. Among the incomplete LU techniques,
multilevel ILU delivers the best balance between robustness and efficiency.
In this section, we give a brief overview of incomplete LU factorization.
We refer readers to \cite{chan1997approximate} for a survey of incomplete
factorization techniques up to 1990s and refer readers to \textcolor{black}{\cite{ghai2017comparison}
for a recent comparison of ILU with other preconditioners for nonsymmetric
systems.}

At a high-level, incomplete LU (or ILU) without pivoting performs
an approximate factorization 
\begin{equation}
\vec{P}^{T}\vec{A}\vec{Q}\approx\vec{L}\vec{U},\label{eq:LU}
\end{equation}
where $\vec{L}$ and $\vec{U}$ are far sparser than their corresponding
factors in the complete LU factorization of $\vec{A}$ after some
row and column reordering. Let $\vec{M}=\vec{L}\vec{U}$, so that
$\vec{P}\vec{M}\vec{Q}^{T}$ is a \emph{preconditioner} of $\vec{A}$,
or equivalently $\vec{M}$ is a preconditioner of $\vec{P}^{T}\vec{A}\vec{Q}$.
In general, $\vec{M}$ is a good preconditioner if the eigenvalues
of $\vec{P}^{T}\vec{A}\vec{Q}\vec{M}$ are well clustered. Given a
linear system $\vec{A}\vec{x}=\vec{b}$, if right-preconditioning
is used, which is typically preferred over left-preconditioning \cite{ghai2017comparison},
the system is then solved by first solving 
\begin{equation}
\vec{A}\left(\vec{P}\vec{M}\vec{Q}^{T}\right)^{-1}\vec{y}=\vec{b},\label{eq:preconditioned-system}
\end{equation}
and then $\vec{x}=\left(\vec{P}\vec{M}\vec{Q}^{T}\right)^{-1}\vec{y}=\vec{Q}\vec{U}^{-1}\vec{L}^{-1}\vec{P}^{T}\vec{y}$.
This type of preconditioner was first used by \textcolor{black}{Simon
\cite{simon1988incomplete} for SPD systems, for which }incomplete
LU reduces to incomplete Cholesky factorization with symmetric reordering,
i.e., $\vec{P}^{T}\vec{A}\vec{P}\approx\vec{R}^{T}\vec{R}$. 

In its simplest form, ILU does not involve any pivoting, and $\vec{L}$
and $\vec{U}$ preserve the sparsity patterns of the lower and upper
triangular parts of $\vec{A}$, respectively. This approach is often
referred to as \emph{ILU0} or \emph{ILU}(0). Unfortunately, it is
typically ineffective and often fails in practice. For linear systems
arising from elliptic PDE, a simple modification, first used in \cite{dupont1968approximate}
and \cite{gustafsson1978class}, is to modify the diagonal entries
to compensate the discarded entries, for example by adding up all
the entries that have been dropped and then subtracting the sum from
the corresponding diagonal entry of $\vec{U}$. This is known as \emph{modified
ILU} \cite{Saad03IMS}.\footnote{The modified ILU is sometimes also abbreviated as MILU. However, in
this work, we use MILU as the abbreviation for multilevel ILU.} A more general approach is to allow \emph{fills}, a.k.a. \emph{fill-ins},
which are new nonzeros in the $\vec{L}$ and $\vec{U}$ factors at
the zeros in $\vec{A}$. Traditionally, the fills are introduced based
on their levels in the elimination tree or based on the magnitude
of numerical values. The former leads to the so-called \emph{ILU}($k$),
which zeros out all the fills of level $k+1$ or higher in the elimination
tree. The combination of the two is known as \emph{ILU with dual thresholding}
(\emph{ILUT}) \cite{saad1994ilut}. Most implementations of ILU, such
as those in PETSc \cite{petsc-user-ref}, hypre \cite{hypre-user},
and the fine-grained parallel ILU algorithm \cite{chow2015fine},
use some variants of ILUT, and they may also allow the user to control
the number of fills in each row.

Because ILUT does not involve pivoting, its effectiveness is still
limited. It often fails in practice for ill-conditioned problems or
for KKT-like systems. To improve its robustness, partial pivoting
can be added into\emph{ }\textit{\emph{ILUT}}, leading to the so-called
\emph{ILUTP} \cite{Saad03IMS}. The\textit{ }ILU implementations in
MATLAB \cite{MATLAB}, SPARSKIT \cite{Saad1994Sparsekit}, and SuperLU
\cite{lishao10}, for example, are based on ILUTP. However, ILUTP
suffers from a major drawback: in general, a small drop tolerance
is needed for ILUTP for robustness, but a large drop tolerance is
needed to avoid rapid growth in the number of fills. As a result,
parameter tuning for ILUTP is a difficult, and sometimes even impossible,
task. Often, the number of fills in ILUTP grows nearly quadratically
for PDE problems as the problem sizes grow, so it is impractical to
use ILUTP for very large-scale problems.

The issues of non-robustness of ILUT and poor-scaling of ILUTP are
mitigated by \emph{multilevel ILU}, or \emph{MILU} for short. Unlike
ILUTP, MILU uses diagonal pivoting instead of partial pivoting with
a deferred factorization for ``problematic'' rows and columns. More
specifically, if a particular row or column would lead to a large
$\left\Vert \vec{L}^{-1}\right\Vert _{\infty}$ or $\left\Vert \vec{U}^{-1}\right\Vert _{1}$,
which are estimated incrementally, then the row and its corresponding
column will be permuted to the lower-right corner of $\vec{A}$ to
be factorized more robustly in the next level \textcolor{black}{\cite{bollhofer2003robust,bollhofer2002relations}}.
Specifically, let $\vec{P}$ and $\vec{Q}$ denote the permutation
matrices due to diagonal pivoting and some reordering. A preconditioner
$\vec{M}$ corresponding to the permuted matrix $\vec{P}^{T}\vec{A}\vec{Q}$
can be constructed via the approximation
\begin{equation}
\vec{P}^{T}\vec{A}\vec{Q}=\begin{bmatrix}\hat{\vec{B}} & \hat{\vec{F}}\\
\hat{\vec{E}} & \hat{\vec{C}}
\end{bmatrix}\approx\hat{\vec{M}}=\begin{bmatrix}\vec{I} & 0\\
\hat{\vec{E}}\tilde{\vec{B}}^{-1} & \vec{I}
\end{bmatrix}\begin{bmatrix}\tilde{\vec{B}} & 0\\
0 & \vec{S}_{C}
\end{bmatrix}\begin{bmatrix}\vec{I} & \tilde{\vec{B}}^{-1}\hat{\vec{F}}\\
0 & \vec{I}
\end{bmatrix},\label{eq:milu-preconditioner}
\end{equation}
where $\tilde{\vec{B}}$ approximates $\hat{\vec{B}}$ via incomplete
factorization, i.e., $\hat{\vec{B}}\approx\tilde{\vec{B}}=\vec{L}\vec{U}$;
$\hat{\vec{C}}$, $\hat{\vec{E}}$ and $\hat{\vec{F}}$ are composed
of the deferred rows and columns; $\vec{S}_{C}$ is its Schur complement.
Here, $\hat{\vec{B}}$ is in general nonsymmetric, unlike $\vec{B}$
in (\ref{eq:predominant-matrix-form}). Note that 
\begin{equation}
\hat{\vec{M}}^{-1}=\begin{bmatrix}\tilde{\vec{B}}^{-1} & \vec{0}\\
\vec{0} & \vec{0}
\end{bmatrix}+\begin{bmatrix}-\tilde{\vec{B}}^{-1}\hat{\vec{F}}\\
\vec{I}
\end{bmatrix}\vec{S}_{C}^{-1}\begin{bmatrix}-\hat{\vec{E}}\tilde{\vec{B}}^{-1} & \vec{I}\end{bmatrix}.\label{eq:Schur-complement-inverse}
\end{equation}
Approximating the inverse of $\vec{S}_{C}$ by the same construction
recursively, we obtain a ``multilevel'' structure. Let $\left(\vec{P}_{S}^{T}\vec{S}_{C}\vec{Q}_{S}\right)^{-1}$
be approximated by $\vec{Q}_{S}^{T}\tilde{\vec{S}}_{C}^{-1}\vec{P}_{S}$,
where $\vec{P}_{S}$ and $\vec{Q}_{S}$ denote the permutation matrices
due to diagonal pivoting and reordering on $\vec{S}_{C}$. The multilevel
preconditioner $\vec{M}$ is defined by
\begin{equation}
\vec{M}^{-1}=\begin{bmatrix}\tilde{\vec{B}}^{-1} & \vec{0}\\
\vec{0} & \vec{0}
\end{bmatrix}+\begin{bmatrix}-\tilde{\vec{B}}^{-1}\hat{\vec{F}}\\
\vec{I}
\end{bmatrix}\tilde{\vec{S}}_{C}^{-1}\begin{bmatrix}-\hat{\vec{E}}\tilde{\vec{B}}^{-1} & \vec{I}\end{bmatrix}.\label{eq:multilevel-inverse}
\end{equation}
Compared to ILUT, MILU is shown to be more robust, thanks to its diagonal
pivoting and deferred factorization. Compared to ILUTP, MILU significantly
reduces the number of fills by avoiding partial pivoting. In addition,
the estimated inverse norm in MILU is also used in the dropping criteria
to further improve robustness. A robust serial implementation of MILU
is available in ILUPACK \textcolor{black}{\cite{Boll06MPC,ilupack}}. 

For virtually all ILU techniques, some preprocessing steps, including
\emph{scaling }(commonly referred to as\emph{ matching}) and \emph{reordering},
can help improve their efficiency and robustness \textcolor{black}{\cite{duff1999design,benzi2000preconditioning}.
}\emph{Scaling} aims to scale the rows and columns of the matrix such
that the diagonal entries have magnitude $1$ and the off-diagonal
entries have magnitude no greater $1$. To this end, the matrix may
need to be permuted using a graph matching algorithm, and hence this
procedure is often referred to as \emph{matching}. A commonly used
scaling procedure is the MC64 routine in the HSL library \cite{duff1999design},
which is based on the maximum weighted bipartite matching algorithm
in \cite{olschowka1996new}. MC64 supports both symmetric and nonsymmetric
matrices. For symmetric matrices, permutation is done symmetrically
to rows and columns, and hence it may be impossible to make all the
diagonal entries to have magnitude $1$. In such cases, the matching
algorithm may need to construct $2\times2$ diagonal blocks whose
off-diagonal entries have magnitude $1$. After matching, one may
perform \emph{reordering} to further reduce fills during factorization.
The popular reordering techniques including approximate minimum degree
(AMD) \cite{amestoy1996approximate}, nested dissection (ND) \cite{georgecomputer},
etc., which are commonly used in sparse LU or Cholesky factorizations.
To preserve the effect of scaling, reordering should be performed
symmetrically. For nonsymmetric matrices, this can be done by applying
the symmetric AMD on $\vec{A}+\vec{A}^{T}$. For symmetric matrices,
a robust reordering algorithm would need to be performed on the block
matrix to preserve the $2\times2$ diagonal blocks.\textcolor{black}{{}
In our experiments, MILU with MC64 matching/scaling and symmetric
AMD reordering tends to deliver very good results, which will be the
basis of this work.}

\section{Predominantly Symmetric Multilevel Incomplete LU-Factorization\label{sec:PS-MILU}}

Our proposed method, \emph{PS-MILU}, aims to exploit the symmetry
in (\ref{eq:predominant-matrix-form}) to provide a unified algorithm
for symmetric and nonsymmetric matrices and to improve efficiency
for predominantly symmetric matrices. To this end, we consider incomplete
LDU factorizations, where the approximate $\vec{L}$ and $\vec{U}$
factors are unit lower and upper triangular matrices, respectively.
For symmetric matrices, $\vec{U}=\vec{L}^{T}$, so LDU factorization
reduces to LDL$^{\text{T}}$ factorization. We use a multilevel framework
with matching, scaling, reordering, and diagonal pivoting at each
level, which will preserve the symmetry of the leading symmetric block. 

Figure~\ref{fig:PS-MILU-Workflow} shows the overall flowchart of
our algorithm. The algorithm takes the input matrix in the form (\ref{eq:predominant-matrix-form}).
If the leading block $\vec{B}$ is symmetric, then we will use incomplete
LDL$^{\text{T}}$ factorization with symmetric matching and reordering.
Note that for best efficiency, the user may need to permute the matrix
beforehand so that the leading symmetric block $\vec{B}$ is as large
as possible. In the context of PDE discretizations, this typically
involves putting the interior nodes before the boundary nodes in the
matrix. The algorithm uses diagonal pivoting to push ``problematic''
rows and columns in $\vec{B}$ to the lower-right corner, which are
then merged into $\vec{C}$, $\vec{E}$ and $\vec{F}$, whose corresponding
Schur complement $\vec{S}_{C}$ is then factorized recursively. The
recursion stops if the Schur complement is small or dense enough.

\begin{figure}
\centering{}\includegraphics[width=0.8\textwidth]{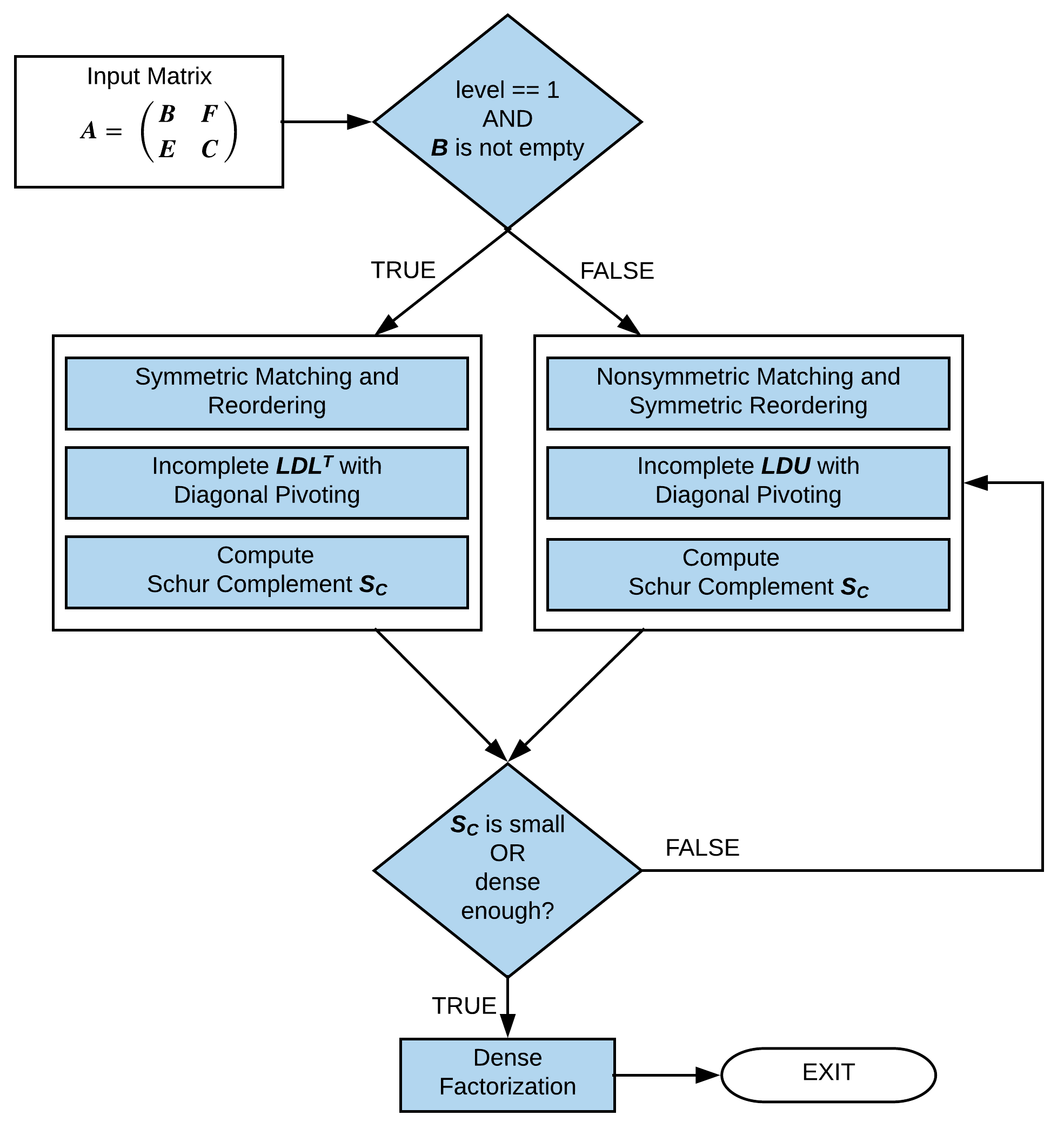}\caption{\label{fig:PS-MILU-Workflow}The overall workflow of PS-MILU.}
\end{figure}

Overall, our algorithm has five main components:
\begin{enumerate}
\item \textbf{preprocessing}, including MC64 matching and AMD reordering
for improving diagonal dominance and reducing fills;
\item \textbf{modified Crout incomplete factorization}, namely, incomplete
symmetric LDL$^{\text{T}}$ and nonsymmetric LDU factorizations, which
adapt the Crout update of $\vec{L}$ and $\vec{U}$ to support early
update of $\vec{D}$ for diagonal pivoting;
\item \textbf{diagonal pivoting}, for controlling the growth of $\left\Vert \vec{L}^{-1}\right\Vert _{\infty}$
and $\left\Vert \vec{U}^{-1}\right\Vert _{1}$;
\item \textbf{inverse-based thresholding}, which controls the dropping in
$\vec{L}$ and $\vec{U}$ to control the growth of the inverse norm;
\item \textbf{hybrid Schur complement}, which modifies the Schur complement
by adaptively adding a correction term to compensate the leading error
term associated with droppings in the current level.
\end{enumerate}
In the following subsections, we will describe each of these components
in more detail.

\subsection{Preprocessing for PS-MILU\label{subsec:Preprocessing-for-PS-MILU}}

In PS-MILU, we apply matching, scaling, and reordering at each level
before factorization. We use the HSL library subroutine MC64 \cite{duff1999design}
followed by symmetric AMD \cite{amestoy1996approximate} for these
purposes. 

Consider the input matrix $\vec{A}$ in (\ref{eq:predominant-matrix-form}),
and assume the leading symmetric block $\vec{B}$ is not empty. We
apply symmetric MC64 on $\vec{B}$, followed by symmetric AMD. Let
\begin{equation}
\vec{A}_{1}=\vec{P}_{1}^{T}\vec{D}_{1}\vec{B}\vec{D}_{1}\vec{P}_{1},\label{eq:preprocessed-B}
\end{equation}
where $\vec{D}_{1}$ and $\vec{P}_{1}$ are the scaling and permutation
matrices, respectively. For now, let us assume that symmetric matching
is successful, so that the diagonals of $\vec{A}_{1}$ are all ones.
Let $\hat{\vec{P}}=\begin{bmatrix}\vec{P}_{1}\\
 & \vec{I}
\end{bmatrix}$ and $\hat{\vec{D}}=\begin{bmatrix}\vec{D}_{1}\\
 & \vec{I}
\end{bmatrix}$, where $\vec{I}$ has the same dimension as $\vec{C}$. Then, we
obtain
\begin{equation}
\hat{\vec{A}}=\hat{\vec{P}}^{T}\hat{\vec{D}}\vec{A}\hat{\vec{D}}\hat{\vec{P}}.\label{eq:preprocessed-matrix-A}
\end{equation}
Figure\ \ref{fig:Sparsity-pattern-FEM} shows the sparsity patterns
of an example predominantly symmetric matrix before and after permutations
due to MC64 and AMD, where the leading block is to the upper-left
of the black lines, whose symmetry is preserved. 
\begin{flushleft}
\begin{figure}
\begin{minipage}[t]{0.45\textwidth}%
\begin{center}
\includegraphics[width=0.9\columnwidth]{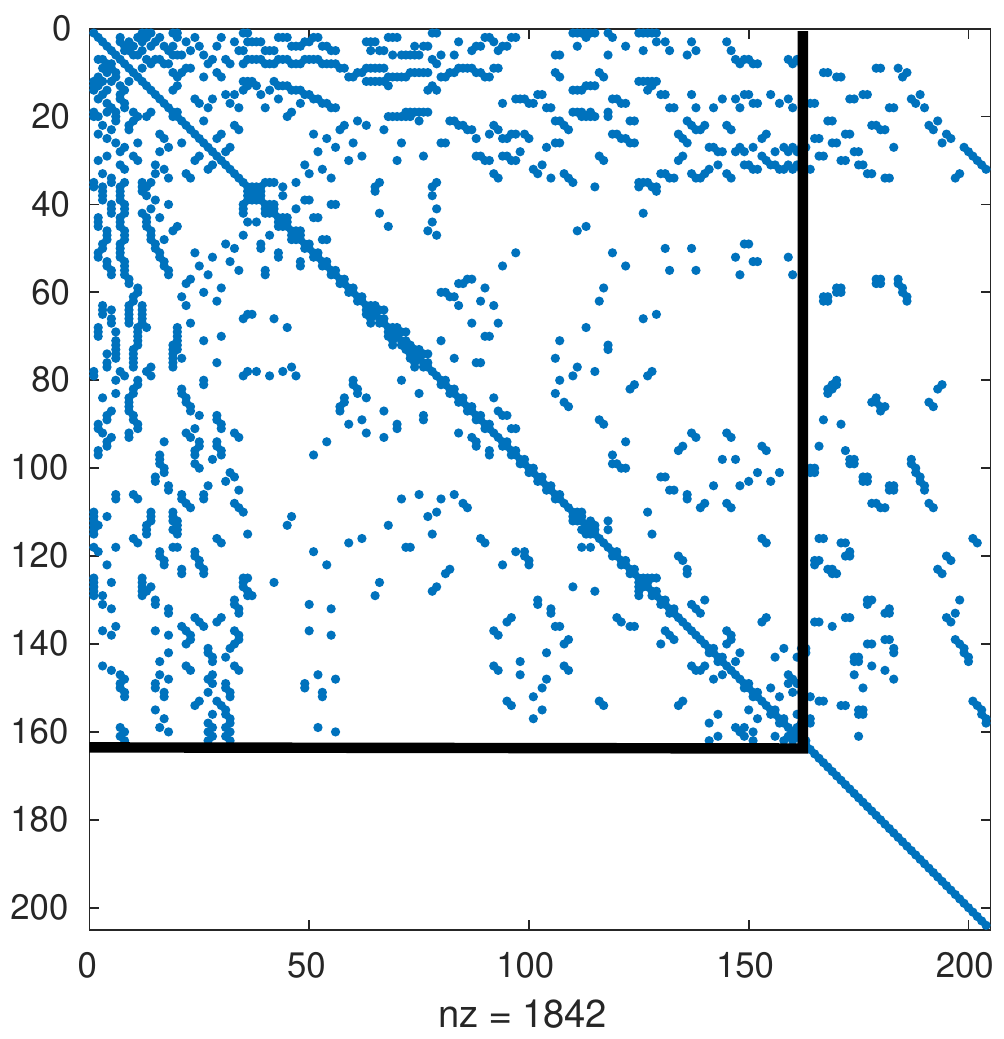}
\par\end{center}%
\end{minipage}\hfill{}%
\begin{minipage}[t]{0.45\textwidth}%
\begin{center}
\includegraphics[width=0.9\columnwidth]{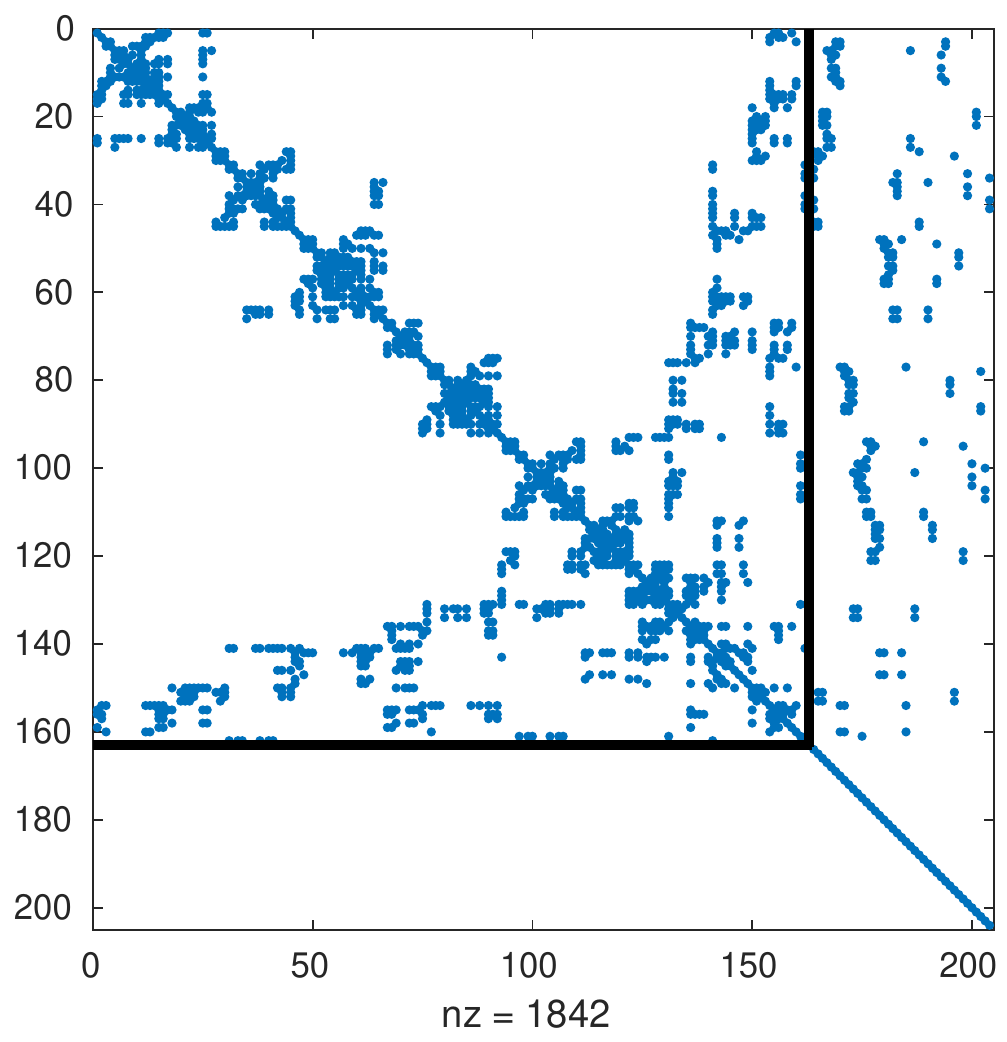}
\par\end{center}%
\end{minipage}

\caption{\label{fig:Sparsity-pattern-FEM} Sparsity pattern of a predominantly
symmetric example matrix from FEM 2D before (left) and after MC64
matching and AMD reordering.}
\end{figure}
\par\end{flushleft}

If $\vec{A}$ is fully nonsymmetric, i.e., $\vec{B}$ is empty, then
we apply the nonsymmetric version of M64 and followed by symmetric
AMD on $\vec{A}+\vec{A}^{T}$ to preserve the scaling effect. In this
case, we have 
\begin{equation}
\hat{\vec{A}}=\vec{P}_{r}^{T}\vec{D}_{r}\vec{A}\vec{D}_{c}\vec{P}_{c},\label{eq:preprocessed-nonsymmetric-A}
\end{equation}
where $\vec{D}_{r}$ and $\vec{D}_{c}$ are row and column scaling
matrices, and $\vec{P}_{r}^{T}$ and $\vec{P}_{c}$ are the row and
column permutation matrices, respectively. Again, all the diagonals
of $\hat{\vec{A}}$ have magnitude 1 due to matching and scaling. 

We note two important special cases. First, symmetric matching and
scaling may result in $2\times2$ modulus-1 diagonal blocks, of which
the diagonal entries may be small or even zero. To take advantage
of these $2\times2$ diagonal blocks, one must perform reordering
on the block matrix followed by a block version of LDL$^{\text{T}}$,
which would significantly complicate the implementation and also introduce
fills into the block rows and columns. For simplicity, we remove the
part of $\vec{B}$ of which the diagonal entries are smaller than
1 after symmetric matching and defer them to the second level, where
nonsymmetric matching will be used. Second, if matrix $\vec{A}$ has
some nearly denser rows and columns, we will permute them to the lower
right corners of $\vec{C}$ and tag them to prevent from diagonal
pivoting. These operations may reduce the size of the block $\vec{A}_{1}$.

\subsection{Incomplete LDL$^{\text{T}}$ and LDU Factorization}

After obtaining the preprocessed matrix $\hat{\vec{A}}$, we then
compute the incomplete factorization. Since the leading symmetric
block may be indefinite, we use incomplete LDL$^{\text{T}}$ factorization
instead of Cholesky factorization for the leading block, and we will
use diagonal pivoting to control the growth of the inverse of the
triangular factors, as described in Section \ref{subsec:Diagonal-Pivoting}.
For a unified treatment of predominantly symmetric and nonsymmetric
matrices, we use LDU factorization with diagonal pivoting, where $\vec{L}$
is unit lower triangular and $\vec{U}$ is unit upper triangular,
and $\vec{U}=\vec{L}^{T}$ if $\hat{\vec{B}}$ is symmetric.

Let $\vec{P}$ denote the permutation matrix due to diagonal pivoting
during the factorization process. The incomplete LDU factorization
computes 
\begin{equation}
\vec{P}^{T}\hat{\vec{A}}\vec{P}=\begin{bmatrix}\hat{\vec{B}} & \hat{\vec{F}}\\
\hat{\vec{E}} & \hat{\vec{C}}
\end{bmatrix}\approx\begin{bmatrix}\tilde{\vec{B}} & \tilde{\vec{F}}\\
\tilde{\vec{E}} & \hat{\vec{C}}
\end{bmatrix}=\begin{bmatrix}\vec{L}_{B} & 0\\
\vec{L}_{E} & \vec{I}
\end{bmatrix}\begin{bmatrix}\vec{D}_{B} & 0\\
0 & \vec{S}_{C}
\end{bmatrix}\begin{bmatrix}\vec{U}_{B} & \vec{U}_{F}\\
0 & \vec{I}
\end{bmatrix},\label{eq:factorization-level-1}
\end{equation}
where, $\hat{\vec{B}}\approx\tilde{\vec{B}}=\vec{L}_{B}\vec{D}_{B}\vec{U}_{B}$,
$\hat{\vec{E}}\approx\tilde{\vec{E}}=\vec{L}_{E}\vec{D}_{B}\vec{U}_{B}$
and $\hat{\vec{F}}\approx\tilde{\vec{F}}=\vec{L}_{B}\vec{D}_{B}\vec{U}_{F}$.
If $\hat{\vec{B}}$ is symmetric, then $\vec{U}_{B}=\vec{L}_{B}^{T}$
and $\tilde{\vec{B}}=\vec{L}_{B}\vec{D}_{B}\vec{L}_{B}^{T}$. Due
to potential diagonal pivoting, $\hat{\vec{B}}$ may have fewer rows
and columns than the leading block $\vec{A}_{1}$ in Section~\ref{subsec:Preprocessing-for-PS-MILU}.
The matrix $\vec{S}_{C}$ is the Schur complement, i.e., $\vec{S}_{C}=\hat{\vec{C}}-\vec{L}_{E}\vec{D}_{B}\vec{U}_{F}$. 

We compute the incomplete LDU factorization using a procedure similar
to the Crout version \cite{li2003crout}. Let $\tilde{\vec{L}}$ and
$\tilde{\vec{U}}$ denote $\begin{bmatrix}\vec{L}_{B}\\
\vec{L}_{E}
\end{bmatrix}$ and $\begin{bmatrix}\vec{U}_{B} & \vec{U}_{F}\end{bmatrix}$, respectively.
Let $\vec{\ell}_{k}$ and $\vec{u}_{k}^{T}$ denote the $k$th column
of $\tilde{\vec{L}}$ and the $k$th row of $\tilde{\vec{U}}$, respectively,
and let $\vec{D}_{k}$ denote the leading $k\times k$ block of $\vec{D}_{B}$.
The Crout version builds $\tilde{\vec{L}}$ and $\tilde{\vec{U}}$
incrementally in columns and rows by computing
\begin{align}
\vec{\ell}_{k} & =\left(\vec{P}^{T}\hat{\vec{A}}\vec{P}\right)_{k+1:n,k}-\tilde{\vec{L}}_{k+1:n,1:k-1}\vec{D}_{k-1}\tilde{\vec{U}}_{1:k-1,k}\label{eq:Crout-Update-L}\\
\vec{u}_{k}^{T} & =\left(\vec{P}^{T}\hat{\vec{A}}\vec{P}\right)_{k,k+1:n}-\tilde{\vec{L}}_{k,1:k-1}\vec{D}_{k-1}\tilde{\vec{U}}_{1:k-1:k+1:n}\label{eq:Crout-Update-U}
\end{align}
at the $k$th step. The left panel of Figure~\ref{fig:Illustration-update}
illustrates the update of $\vec{\ell}_{k}$; and the update of $\vec{u}_{k}^{T}$
is done symmetrically. The right panel of Figure~\ref{fig:Illustration-update}
illustrates the details of updating and dropping, which we will explain
in Section~\ref{subsec:Overall-Factorization-Algorithm}. In PS-MILU,
the entries of $\vec{u}_{k}^{T}$ within $\vec{U}_{B}$ are equal
to their corresponding entries in $\vec{\ell}_{k}$, so we can save
their computational cost by half. However, those in $\vec{U}_{F}$
must be updated separately from $\vec{L}_{E}$. If $\hat{\vec{A}}$
is fully nonsymmetric, the same procedure described above applies
by updating the whole $\vec{u}_{k}^{T}$.

\begin{figure}
\begin{centering}
\includegraphics[width=1\textwidth]{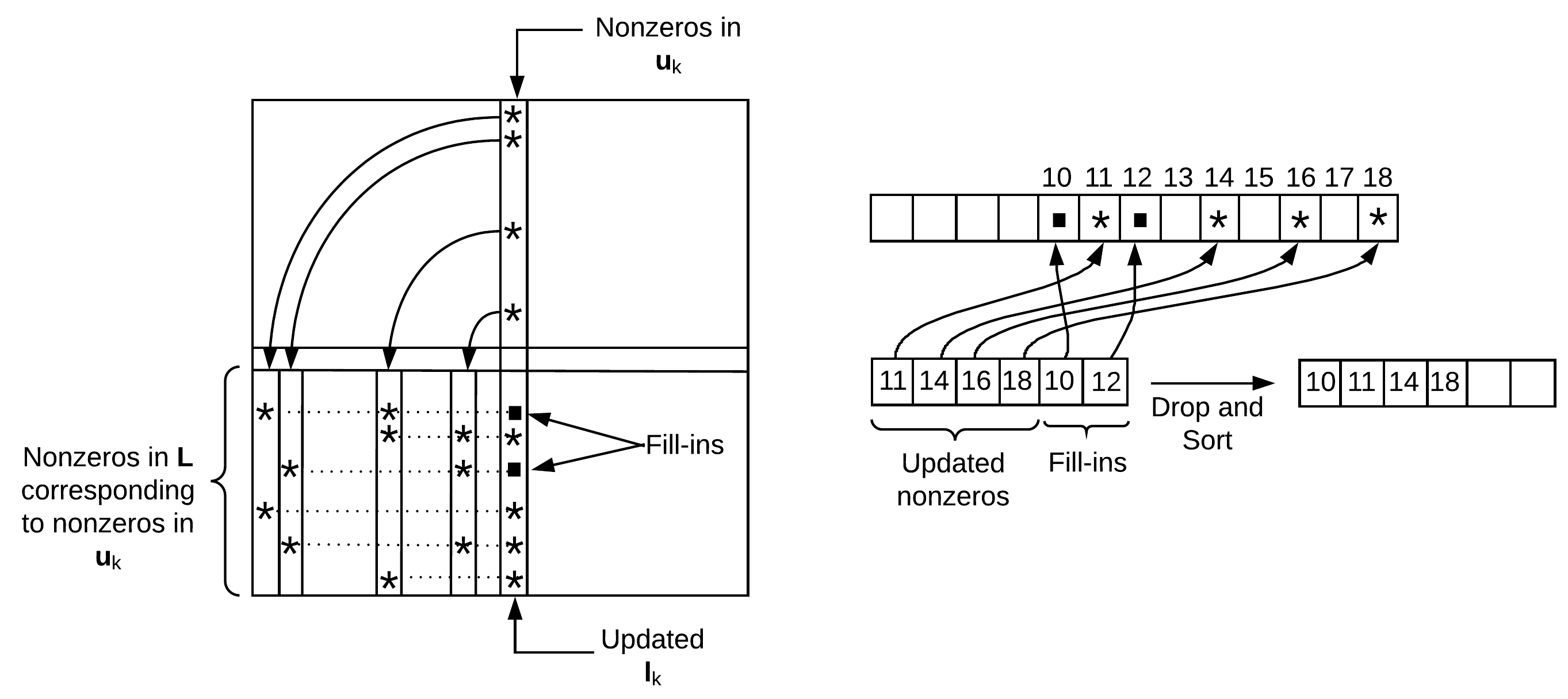}
\par\end{centering}
\caption{\label{fig:Illustration-update}Illustration of updating $\vec{\ell}_{k}$
in Crout update.}
\end{figure}

Compared to the classical LU factorization procedures, which update
the Schur complement at the $k$th step, the Crout version allows
easier incorporation of dropping in $\vec{\ell}_{k}$ and $\vec{u}_{k}^{T}$,
as we will describe shortly. The Crout procedure in \cite{li2003crout}
did not support diagonal pivoting. For effective pivoting, we update
$\vec{D}_{B}$ after computing $\vec{\ell}_{k}$ and $\vec{u}_{k}^{T}$
at step $k$, before applying droppings to $\vec{\ell}_{k}$ and $\vec{u}_{k}^{T}$.
Specifically, let $\vec{d}^{(k-1)}$ denote the vector containing
the partial result of $\vec{D}_{B}$ from the step $k-1$, where $\vec{d}^{(0)}$
is initialized to the diagonal entries of $\hat{\vec{B}}$. Let $\vec{\ell}_{B,k}$
and $\vec{u}_{B,k}^{T}$ denote the $\vec{L}_{B}$ and $\vec{U}_{B}$
portions of $\vec{\ell}_{k}$ and $\vec{u}_{k}^{T}$, respectively.
Then,
\[
\vec{d}_{k+1:m}^{(k)}=\vec{d}_{k+1:m}^{(k-1)}-d_{k}\vec{\ell}_{B,k}.*\vec{u}_{B,k}^{T},
\]
where $m$ is the size of $\hat{\vec{B}}$ and $.*$ denotes element-wise
multiplication. When diagonal pivoting is performed, we will need
to permute entries in $\vec{d}$, along with the rows in $\tilde{\vec{L}}$
and columns in $\tilde{\vec{U}}$, which we describe next.

\subsection{Diagonal Pivoting\label{subsec:Diagonal-Pivoting}}

One of the crucial components in PS-MILU is diagonal pivoting. For
nonsymmetric systems, diagonal pivoting has been shown to be very
effective \cite{bollhofer2003robust}. For symmetric systems, the
importance of pivoting is at least as important as for nonsymmetric
cases, for at least two reasons. First, as a direct method, LDL$^{\text{T}}$
without pivoting may break down for indefinite systems. A simplest
example is $\vec{A}=\begin{bmatrix}0 & 1\\
1 & 0
\end{bmatrix}.$ A direct implication of this fact is that incomplete LDL$^{\text{T}}$
without pivoting can suffer from instability for symmetric and indefinite
systems. A well-known factorization for symmetric and indefinite systems
is block LDL$^{\text{T}}$ with diagonal pivoting \cite{bunch1971direct},
which requires $1\times1$ and $2\times2$ diagonal blocks. Even though
it may be possible that block LDL with much larger blocks can be stable
without pivoting across blocks, but pivoting may still be needed within
blocks.

The second reason is that for incomplete factorizations, diagonal
pivoting is important in controlling the growth of the inverse of
the triangular factors. This can be shown as follows. Consider the
incomplete LDU factorization with pivoting,
\begin{equation}
\vec{P}^{T}\hat{\vec{A}}\vec{Q}=\vec{L}\vec{D}\vec{U}+\vec{\delta}_{A},\label{eq:ILU-error}
\end{equation}
where $\vec{\delta}_{A}$ is the error due to dropping. Let $\vec{M}=\vec{L}\vec{D}\vec{U}$
be a right-preconditioner of $\vec{P}^{T}\hat{\vec{A}}\vec{Q}$. Then,
\begin{equation}
\vec{P}^{T}\hat{\vec{A}}\vec{Q}\vec{M}^{-1}=\vec{I}+\vec{\delta}_{A}\vec{M}^{-1}=\vec{I}+\vec{\delta}_{A}\vec{U}^{-1}\vec{D}^{-1}\vec{L}^{-1},\label{eq:right-preconditioned-matrix}
\end{equation}
where the spectral radius of the second term is bounded by 
\begin{align}
\rho\left(\vec{\delta}_{A}\vec{M}^{-1}\right) & \leq\left\Vert \vec{\delta}_{A}\vec{U}^{-1}\vec{D}^{-1}\vec{L}^{-1}\right\Vert \label{eq:error-bounds}\\
 & \,\leq\left\Vert \vec{D}^{-1}\right\Vert \left\Vert \vec{L}^{-1}\right\Vert \left\Vert \vec{U}^{-1}\right\Vert \left\Vert \vec{\delta}_{A}\right\Vert .\label{eq:condition-number-spectrum}
\end{align}
Therefore, $\left\Vert \vec{D}^{-1}\right\Vert \left\Vert \vec{L}^{-1}\right\Vert \left\Vert \vec{U}^{-1}\right\Vert $
is an absolute condition number for the spectral radius of the preconditioned
matrix $\vec{P}^{T}\hat{\vec{A}}\vec{Q}\vec{M}^{-1}$ with respect
to $\left\Vert \vec{\delta}_{A}\right\Vert $. For a well-scaled matrix
$\hat{\vec{A}}$, $\left\Vert \vec{\delta}_{A}\right\Vert $ is well
bounded, and $\left\Vert \vec{D}^{-1}\right\Vert $ is typically small.
However, a large $\left\Vert \vec{L}^{-1}\right\Vert $ or $\left\Vert \vec{U}^{-1}\right\Vert $
can significantly deteriorate the spectral radius, which in turn can
undermine the convergence of the preconditioned KSP method. Note that
even with block LDL$^{\text{T}}$ without pivoting, $\left\Vert \vec{L}^{-1}\right\Vert $
and $\left\Vert \vec{U}^{-1}\right\Vert $ may still grow rapidly.
Therefore, we consider pivoting indispensable for the robustness of
incomplete factorization.

Estimating the 2-norms are relatively expensive. However, we can estimate
and bound the $\infty$-norm of the inverse of a triangular matrix
efficiently, in that 
\begin{equation}
\left\Vert \vec{T}^{-1}\right\Vert _{\infty}=\sup_{\Vert\vec{v}\Vert_{\infty}=1}\left\Vert \vec{T}^{-1}\vec{v}\right\Vert _{\infty}\gtrapprox\sup_{c_{i}=\pm1}\left\Vert \vec{T}^{-1}\vec{c}\right\Vert _{\infty},\label{eq:triangular-norm}
\end{equation}
where the sign of $c_{i}$ is chosen in a greedy fashion to maximize
$\left|\vec{e}_{i}^{T}\vec{T}^{-1}\vec{c}\right|$; see \cite{cline1979estimate,golub2012matrix}
for more detail. In the context of LDU factorization, let $\vec{L}_{k}$
and $\vec{U}_{k}$ denote the leading $k\times k$ blocks of $\vec{L}$
and $\vec{U}$, respectively, and let $\vec{c}_{L,k}$ and $\vec{c}_{U,k}$
denote their corresponding $\vec{c}$ vectors in (\ref{eq:triangular-norm}).
We can bound $\left\Vert \vec{L}_{k}^{-1}\right\Vert _{\infty}$ incrementally:
Given that $\left\Vert \vec{L}_{k-1}^{-1}\vec{c}_{L,k-1}\right\Vert _{\infty}\leq\tau_{\kappa}$
for some threshold $\tau_{\kappa}$ as estimated greedily based on
(\ref{eq:triangular-norm}), then $\left\Vert \vec{L}_{k}^{-1}\vec{c}_{L,k}\right\Vert _{\infty}>\tau_{\kappa}$
based on the same estimation if and only if
\begin{equation}
\tilde{\kappa}_{L,k}=\left|\vec{e}_{k}^{T}\vec{L}_{k}^{-1}\vec{c}_{L,k}\right|>\tau_{k},\label{eq:bound inv L}
\end{equation}
where $\vec{e}_{k}\in\mathbb{R}^{k}$. Similarly, we can bound $\left\Vert \vec{U}_{k}^{-1}\right\Vert _{1}=\left\Vert \vec{U}_{k}^{-T}\right\Vert _{\infty}$
by estimating 
\begin{equation}
\tilde{\kappa}_{U,k}=\left|\vec{e}_{k}^{T}\vec{U}_{k}^{-T}\vec{c}_{U,k}\right|,\label{eq:bound inv U}
\end{equation}
where $\tilde{\kappa}_{U,k}=\tilde{\kappa}_{L,k}$ for symmetric $\hat{\vec{B}}$. 

It is easy to incorporate the estimation of the condition numbers
into the Crout update. In particular, after computing $\vec{\ell}_{k}$
and $\vec{u}_{k}^{T}$ at the $k$th step, we incrementally update
$\vec{c}_{L,k}$, $\vec{c}_{U,k}$, $\vec{L}_{k}^{-1}\vec{c}_{L,k}$
and $\vec{U}_{k}^{-T}\vec{c}_{U,k}$ from their partial results in
step $k-1$, compute $\tilde{\kappa}_{L,k}$ and $\tilde{\kappa}_{U,k}$,
and then compare them with the threshold $\tau_{\kappa}$. According
to (\ref{eq:condition-number-spectrum}), we also need to safeguard
$\left\Vert \vec{D}^{-1}\right\Vert $ by comparing $\left|1/d_{k}\right|$
against a threshold $\tau_{d}$. Let $m$ denote the size of $\hat{\vec{B}}$.
If $\max\{\tilde{\kappa}_{L,k},\tilde{\kappa}_{U,k}\}>\tau_{\kappa}$,
or $\vert1/d_{k}\vert>\tau_{d}$, we exchange row $k$ and column
$k$ in $\hat{\vec{A}}$ with the $m$th row and column of $\hat{\vec{A}}$,
reduce the size of $\hat{\vec{B}}$ (i.e., $m$) by one, and recompute
the Crout update. Figure~\ref{fig:Diagonal-pivoting} illustrates
this pivot operation. This process repeats until $\tilde{\kappa}_{L,k}$,
$\tilde{\kappa}_{U,k}$, and $\vert1/d_{k}\vert$ are all within threshold
or the leading block has been exhausted. Since we have pre-updated
$\vec{D}$ in the modified Crout update, we check $\vert1/d_{m}\vert$
before pivoting and reduce $m$ by $1$ if $\vert1/d_{m}\vert<\tau_{d}$
directly. This avoids unnecessary exchanges if $\left|d_{m}\right|$
is small. In terms of implementation, special care must be taken to
ensure optimal time complexity, which we will address in Section\ \ref{sec:Implementation-details}.
\begin{figure}
\begin{centering}
\includegraphics[width=1\textwidth]{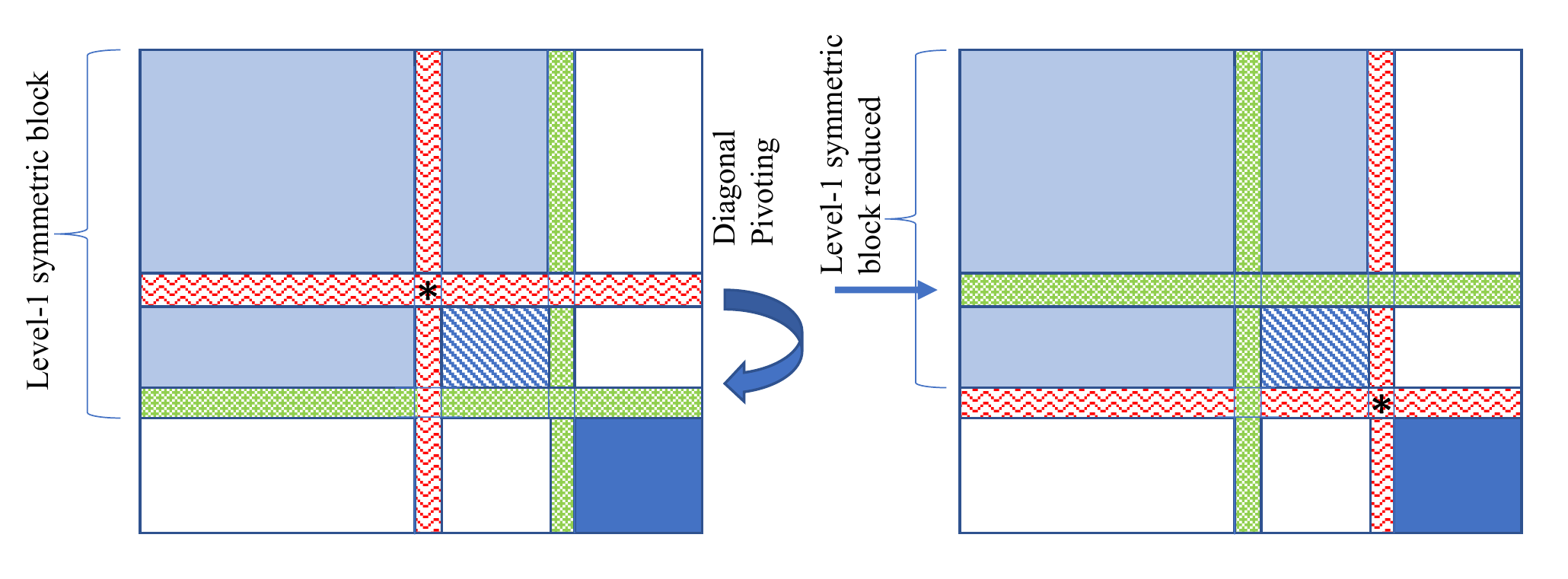}
\par\end{centering}
\caption{{\small{}\label{fig:Diagonal-pivoting}Illustration of diagonal pivoting
for level-$1$ symmetric block. It removes }``problematic'' rows
and columns from the current level and merge them with the block for
more robust processing in the next level.}
\end{figure}

\subsection{Inverse-Based Thresholding}

For incomplete factorization to be effective, it must balance two
important factors: first, it must control the numbers of nonzeros
in $\vec{L}$ and $\vec{U}$, which ideally should be linear in the
input size. This requires dropping ``insignificant'' nonzeros in
$\vec{L}$ and $\vec{U}$. Second, as shown above, it is important
to control $\left\Vert \vec{D}^{-1}\right\Vert \left\Vert \vec{L}^{-1}\right\Vert \left\Vert \vec{U}^{-1}\right\Vert $,
which determine the spectral radius of the preconditioned matrix.
For this reason, the traditional dropping criteria, such as those
based on the level in the elimination tree or the magnitude of the
entries, are often ineffective. A robust approach is the inverse-based
thresholding proposed in \cite{bollhofer2003robust}. We adopt this
approach, but we pay special attention to achieve optimal time complexity
in terms of the number of fills and the computational cost.

The inverse-based thresholding can be motivated by an argument similar
to, but more detailed than, the analysis in the preceding subsection.
In particular, let 
\begin{align}
\vec{P}^{T}\hat{\vec{A}}\vec{Q} & =\left(\vec{L}+\vec{\delta}_{L}\right)\left(\vec{D}+\vec{\delta}_{D}\right)\left(\vec{U}+\vec{\delta}_{U}\right)\nonumber \\
 & =\vec{L}\vec{D}\vec{U}+\vec{\delta}_{L}\vec{D}\vec{U}+\vec{L}\vec{D}\vec{\delta}_{U}+\vec{L}\vec{\delta}_{D}\vec{U}+\text{h.o.t.},\label{eq:dropping-terms}
\end{align}
where we omit the higher-order terms that involve more than one $\vec{\delta}$
matrix. Let $\vec{M}=\vec{L}\vec{D}\vec{U}$ be a right preconditioner
of $\vec{P}^{T}\hat{\vec{A}}\vec{Q}$. Similar to (\ref{eq:error-bounds}),
\begin{align}
\rho\left(\vec{\delta}_{A}\vec{M}^{-1}\right) & \approx\rho\left(\left(\vec{\delta}_{L}\vec{D}\vec{U}+\vec{L}\vec{D}\vec{\delta}_{U}+\vec{L}\vec{\delta}_{D}\vec{U}\right)\vec{M}^{-1}\right)\nonumber \\
 & \leq\left\Vert \vec{\delta}_{L}\vec{L}^{-1}+\vec{L}\vec{D}\vec{\delta}_{U}\vec{U}^{-1}\vec{D}^{-1}\vec{L}^{-1}+\vec{L}\vec{\delta}_{D}\vec{D}^{-1}\vec{L}^{-1}\right\Vert \nonumber \\
 & \leq\left\Vert \vec{\delta}_{L}\vec{L}^{-1}\right\Vert +\kappa(\vec{M})\left\Vert \vec{U}^{-1}\vec{\delta}_{U}\right\Vert +\kappa(\vec{L})\left\Vert \vec{\delta}_{D}\vec{D}^{-1}\right\Vert .\label{eq:spectral-radius-bound}
\end{align}
When deciding whether to drop a specific $\ell_{ik}$, consider $\vec{\delta}{}_{ik}=\ell_{ik}\hat{\vec{e}}_{i}\hat{\vec{e}}_{k}^{T}$,
where $i\geq k$ and $\hat{\vec{e}}_{i},\hat{\vec{e}}_{k}\in\mathbb{R}^{n}$.
Then,
\[
\left\Vert \vec{\delta}{}_{ik}\vec{L}^{-1}\right\Vert =\left\Vert \ell_{ik}\hat{\vec{e}}_{i}\hat{\vec{e}}_{k}^{T}\vec{L}^{-1}\right\Vert =\left|\ell_{ik}\right|\left\Vert \hat{\vec{e}}_{k}^{T}\vec{L}^{-1}\right\Vert =\left|\ell_{ik}\right|\left\Vert \vec{e}_{k}^{T}\vec{L}_{k}^{-1}\right\Vert ,
\]
where $\vec{e}_{k}\in\mathbb{R}^{k}$. Instead of estimating the 2-norm,
we can estimate the $\infty$-norm as 
\[
\left\Vert \vec{e}_{k}^{T}\vec{L}_{k}^{-1}\right\Vert _{\infty}=\sup_{c_{i}=\pm1}\left|\vec{e}_{k}^{T}\vec{L}_{k}^{-1}\vec{c}\right|\gtrapprox\left|\vec{e}_{k}^{T}\vec{L}_{k}^{-1}\vec{c}_{L,k}\right|=\tilde{\kappa}_{L,k},
\]
where $\vec{c}_{L,k}$ and $\tilde{\kappa}_{L,k}$ were defined in
Section~\ref{subsec:Diagonal-Pivoting}. Similarly, let $\vec{\delta}{}_{kj}=u_{kj}\hat{\vec{e}}_{k}\hat{\vec{e}}_{j}^{T}$,
we then have 
\[
\left\Vert \vec{U}^{-1}\vec{\delta}{}_{kj}\right\Vert =\left\Vert \vec{U}^{-1}\left(u_{ki}\hat{\vec{e}}_{k}\hat{\vec{e}}_{j}^{T}\right)\right\Vert =\left|u_{kj}\right|\left\Vert \vec{e}_{k}^{T}\vec{U}_{k}^{-T}\right\Vert ,
\]
and $\left\Vert \vec{e}_{k}^{T}\vec{U}_{k}^{-T}\right\Vert _{\infty}\gtrapprox\left|\vec{e}_{k}^{T}\vec{U}_{k}^{-T}\vec{c}_{U,k}\right|=\tilde{\kappa}_{U,k}$.
Therefore, as a heuristic, we drop $\ell_{ik}$ if
\[
\left|\ell_{ik}\right|\tilde{\kappa}_{L,k}\leq\tau_{L}
\]
 and drop $u_{kj}$ if 
\[
\left|u_{kj}\right|\tilde{\kappa}_{U,k}\leq\tau_{U}
\]
for some $\tau_{L}$ and $\tau_{U}$. This is referred to as the \emph{inverse-based
thresholding} \cite{li2003crout,bollhofer2003robust}. To control
the time complexity of the algorithm, we also limit the numbers of
nonzeros in $\vec{\ell}_{k}$ and $\vec{u}_{k}^{T}$ to be within
constant factors (specifically, $\alpha_{L}$ and $\alpha_{U}$) of
the numbers of nonzeros in the corresponding row and column in $\vec{P}^{T}\hat{\vec{A}}\vec{Q}$,
and drop the excess nonzeros even if their values are above the threshold
$\tau_{L}/\tilde{\kappa}_{L,k}$ and $\tau_{U}/\tilde{\kappa}_{U,k}$.

This dropping technique can be easily incorporated into the modified
Crout update procedure, in that $\tilde{\kappa}_{L,k}$ and $\tilde{\kappa}_{U,k}$
have already been computed. According to (\ref{eq:spectral-radius-bound}),
assuming $\vec{M}$ is used as a right-preconditioner, the spectral
radius is the most sensitive to $\vec{\delta}_{U}$, followed by $\vec{\delta}_{D}$
and $\vec{\delta}_{L}$ in that order. To reduce $\vec{\delta}_{D}$,
we update $\vec{D}$ before applying dropping to $\vec{\ell}_{k}$
and $\vec{u}_{k}^{T}$ in the modified Crout update, as we mentioned
Section~\ref{subsec:Diagonal-Pivoting}. Furthermore, to compensate
the sensitivity of $\vec{\delta}_{U}$, it is advantageous for $\tau_{U}\leq\tau_{L}$
and $\alpha_{U}\geq\alpha_{L}$. For a symmetric leading block, we
let $\tau_{L}=\tau_{U}$ and $\alpha_{L}=\alpha_{U}$. 

\subsection{Hybrid Schur Complement}

Another key component is the computation of the Schur complement,
which will be factorized in the next level. From (\ref{eq:factorization-level-1}),
the Schur complement is defined as
\begin{align}
\vec{S}_{C} & =\left[\begin{array}{cc}
-\vec{L}_{E}\vec{L}_{B}^{-1} & \vec{I}\end{array}\right]\begin{bmatrix}\tilde{\vec{B}} & \tilde{\vec{F}}\\
\tilde{\vec{E}} & \hat{\vec{C}}
\end{bmatrix}\left[\begin{array}{c}
-\vec{U}_{B}^{-1}\vec{U}_{F}\\
\vec{I}
\end{array}\right]\label{eq:Schur complement S-version}\\
 & =\hat{\vec{C}}-\vec{L}_{E}\vec{D}_{B}\vec{U}_{F}.
\end{align}
Note that $\hat{\vec{C}}$ was permuted and scaled, but it involves
no dropping. However, $\tilde{\vec{B}}$, $\tilde{\vec{E}}$, and
$\tilde{\vec{F}}$ do involve dropping, and the above definition does
not take into account the effect of dropping on the preconditioned
matrix. Let $\vec{M}^{-1}=\begin{bmatrix}\tilde{\vec{B}} & \tilde{\vec{F}}\\
\tilde{\vec{E}} & \hat{\vec{C}}
\end{bmatrix}$, and $\vec{P}^{T}\hat{\vec{A}}\vec{P}\vec{M}^{-1}=\begin{bmatrix}\hat{\vec{B}} & \hat{\vec{F}}\\
\hat{\vec{E}} & \hat{\vec{C}}
\end{bmatrix}\begin{bmatrix}\tilde{\vec{B}} & \tilde{\vec{F}}\\
\tilde{\vec{E}} & \hat{\vec{C}}
\end{bmatrix}^{-1}$. The errors in $\vec{\delta}_{B}=\hat{\vec{B}}-\tilde{\vec{B}}$
can be magnified by both $\vec{L}_{B}^{-1}$ and $\vec{U}_{B}^{-1}$,
as we shall show shortly. Hence, we introduce a modified formula for
the Schur complement, 
\begin{align}
\vec{H}_{C} & =\vec{C}-2\vec{L}_{E}\vec{D}_{B}\vec{U}_{F}+\vec{L}_{E}\vec{L}_{B}^{-1}\hat{\vec{B}}\vec{U}_{B}^{-1}\vec{U}_{F}\label{eq:modified-Schur-complement}\\
 & =\vec{S}_{C}+\vec{L}_{E}\left(\vec{L}_{B}^{-1}\hat{\vec{B}}\vec{U}_{B}^{-1}-\vec{D}_{B}\right)\vec{U}_{F},
\end{align}
where the second term eliminates this error term due to $\vec{\delta}_{B}$.

To derive (\ref{eq:modified-Schur-complement}), let us first define
a preconditioner $\hat{\vec{M}}$ as
\[
\hat{\vec{M}}=\begin{bmatrix}\vec{I} & 0\\
\vec{L}_{E}\vec{L}_{B}^{-1} & \vec{I}
\end{bmatrix}\begin{bmatrix}\tilde{\vec{B}} & 0\\
0 & \vec{T}_{C}
\end{bmatrix}\begin{bmatrix}\vec{I} & \vec{U}_{B}^{-1}\vec{U}_{F}\\
0 & \vec{I}
\end{bmatrix},
\]
where $\vec{T}_{C}$ has yet to be defined. Then, 
\begin{align*}
\rho\left(\vec{P}^{T}\hat{\vec{A}}\vec{P}\hat{\vec{M}}^{-1}\right) & =\rho\left(\begin{bmatrix}\vec{I} & 0\\
\vec{L}_{E}\vec{L}_{B}^{-1} & \vec{I}
\end{bmatrix}^{-1}\begin{bmatrix}\hat{\vec{B}} & \hat{\vec{F}}\\
\hat{\vec{E}} & \hat{\vec{C}}
\end{bmatrix}\begin{bmatrix}\vec{I} & \vec{U}_{B}^{-1}\vec{U}_{F}\\
0 & \vec{I}
\end{bmatrix}^{-1}\begin{bmatrix}\tilde{\vec{B}} & 0\\
0 & \vec{T}_{C}
\end{bmatrix}^{-1}\right)\\
 & =\rho\left(\begin{bmatrix}\vec{I} & 0\\
-\vec{L}_{E}\vec{L}_{B}^{-1} & \vec{I}
\end{bmatrix}\begin{bmatrix}\hat{\vec{B}} & \hat{\vec{F}}\\
\hat{\vec{E}} & \hat{\vec{C}}
\end{bmatrix}\begin{bmatrix}\vec{I} & -\vec{U}_{B}^{-1}\vec{U}_{F}\\
0 & \vec{I}
\end{bmatrix}\begin{bmatrix}\tilde{\vec{B}} & 0\\
0 & \vec{T}_{C}
\end{bmatrix}^{-1}\right)\\
 & =\rho\left(\begin{bmatrix}\hat{\vec{B}} & \hat{\vec{\delta}}_{F}\\
\hat{\vec{\delta}}_{E} & \hat{\vec{T}}_{C}
\end{bmatrix}\begin{bmatrix}\tilde{\vec{B}} & 0\\
0 & \vec{T}_{C}
\end{bmatrix}^{-1}\right).
\end{align*}
where $\hat{\vec{\delta}}_{E}$ and $\hat{\vec{\delta}}_{F}$ correspond
to the effect of dropping in $\hat{\vec{E}}$ and $\hat{\vec{F}}$,
respectively. Suppose $\hat{\vec{\delta}}_{E}$ and $\hat{\vec{\delta}}_{F}$
are small, and the above spectral radius is approximately minimized
if $\hat{\vec{T}}_{C}\vec{T}_{C}^{-1}=\vec{I}$, i.e.,
\begin{align}
\vec{T}_{C}=\hat{\vec{T}}_{C} & =\left[\begin{array}{cc}
-\vec{L}_{E}\vec{L}_{B}^{-1} & \vec{I}\end{array}\right]\begin{bmatrix}\hat{\vec{B}} & \hat{\vec{F}}\\
\hat{\vec{E}} & \hat{\vec{C}}
\end{bmatrix}\left[\begin{array}{c}
-\vec{U}_{B}^{-1}\vec{U}_{F}\\
\vec{I}
\end{array}\right]\label{eq:Schur complement T-version}\\
 & =\vec{L}_{E}\vec{L}_{B}^{-1}\hat{\vec{B}}\vec{U}_{B}^{-1}\vec{U}_{F}-\hat{\vec{E}}\vec{U}_{B}^{-1}\vec{U}_{F}-\vec{L}_{E}\vec{L}_{B}^{-1}\hat{\vec{F}}+\hat{\vec{C}}.
\end{align}
This formulation of Schur complement was introduced in \cite{tismenetsky1991new}
and was also used in \cite{Boll06MPC}. Bollhöfer and Saad \cite{Boll06MPC}
referred to (\ref{eq:Schur complement S-version}) and (\ref{eq:Schur complement T-version})
as the S-version and T-version, respectively. However, $\vec{T}_{C}$
is relatively complicated. Note that 
\begin{align}
\vec{T}_{C}-\vec{S}_{C} & =\left[\begin{array}{cc}
-\vec{L}_{E}\vec{L}_{B}^{-1} & \vec{I}\end{array}\right]\begin{bmatrix}\vec{\delta}_{B} & \vec{\delta}_{F}\\
\vec{\delta}_{E} & \vec{0}
\end{bmatrix}\left[\begin{array}{c}
-\vec{U}_{B}^{-1}\vec{U}_{F}\\
\vec{I}
\end{array}\right]\nonumber \\
 & =\vec{L}_{E}\vec{L}_{B}^{-1}\vec{\delta}_{B}\vec{U}_{B}^{-1}\vec{U}_{F}-\vec{\delta}_{E}\vec{U}_{B}^{-1}\vec{U}_{F}-\vec{L}_{E}\vec{L}_{B}^{-1}\vec{\delta}_{F}.\label{eq:error-Schur-complement}
\end{align}
Note that $\vec{\delta}_{E}$ and $\vec{\delta}_{F}$ are multiplied
by $\vec{U}_{B}^{-1}$ and $\vec{L}_{B}^{-1}$, respectively, so their
effects are similar to the droppings in $\vec{L}_{B}$ in (\ref{eq:spectral-radius-bound}).
However, $\vec{\delta}_{B}$ is multiplied by both $\vec{L}_{B}^{-1}$
and $\vec{U}_{B}^{-1}$, which can ``square the condition number.''
Eq.~(\ref{eq:modified-Schur-complement}) adds the leading term in
(\ref{eq:error-Schur-complement}) to $\vec{S}_{C}$ to obtain $\vec{H}_{C}$
and in turn avoids the ``squaring'' effect. Note that 
\[
\vec{H}_{C}=\left[\begin{array}{cc}
-\vec{L}_{E}\vec{L}_{B}^{-1} & \vec{I}\end{array}\right]\begin{bmatrix}\hat{\vec{B}} & \tilde{\vec{F}}\\
\tilde{\vec{E}} & \hat{\vec{C}}
\end{bmatrix}\left[\begin{array}{c}
-\vec{U}_{B}^{-1}\vec{U}_{F}\\
\vec{I}
\end{array}\right],
\]
which can be viewed as a hybrid of $\vec{S}_{C}$ and $\vec{T}_{C}$
in (\ref{eq:Schur complement S-version}) and (\ref{eq:Schur complement T-version}),
respectively. Therefore, we refer to $\vec{H}_{C}$ as the \emph{hybrid
version}, or \emph{H-version}, of the Schur complement. The H-version
has similar accuracy as the T-version, and it is simpler and can be
computed as an update to $\vec{S}_{C}$.

In both the H- and the T-version, the term $\vec{L}_{E}\vec{L}_{B}^{-1}\hat{\vec{B}}\vec{U}_{B}^{-1}\vec{U}_{F}$
can potentially make the Schur complement dense. Hence, we use the
S-version without dropping for almost all levels. Let $n_{C}$ denote
the size of $\hat{\vec{C}}$ at a particular level. If the Schur complement
is small (in particular, $n_{C}=\mathcal{O}(n^{1/3})$) or the S-version
is nearly dense, then we would apply a dense complete LU factorization
with column pivoting to factorize it; otherwise, we factorize $\vec{S}_{C}$
using multilevel ILU recursively. We update the Schur complement to
the H-version if $n_{C}$ is smaller than a user-specified $c_{h}$.
We will address the complexity analysis in more detail in Section~\ref{subsec:Time-Complexity}.
After obtaining the multilevel ILU, one can then use (\ref{eq:multilevel-inverse})
recursively to construct a preconditioner.

\section{Implementation Details and Complexity Analysis\label{sec:Implementation-details}}

The preceding section focused on the robustness of PS-MILU. In this
section, we address its efficient implementation, to achieve optimal
time complexity. More specifically, if the number of nonzeros per
row and per column are bounded by a constant, the algorithm and its
implementation must scale linearly with respect to the input size.
To the best our knowledge, the only methods in the literature that
could achieve linear-time complexity are ILU0 or ILUT without pivoting,
which unfortunately are not robust. As demonstrated in \cite{ghai2017comparison},
multilevel ILU scales nearly linearly for certain classes of problems,
but there was no theoretical complexity analysis in the literature.

To achieve optimal complexity for PS-MILU, we must have a data structure
for sparse matrices that supports efficient sequential access of a
matrix in amortized constant time per nonzero in both row \textsl{and}
column major, while supporting efficient row or column interchanges.
In addition, the costs of thresholding, pivoting, and sorting must
not dominate the intrinsic floating-point operations in the incomplete
LU factorization. In the following, we will describe our data structures,
present the pseudocode, and then prove the optimal time complexity
of our algorithm.

\subsection{An Augmented Sparse Matrix Storage\label{subsec:Data-Structure}}

To facilitate the implementation of PS-MILU, we need an efficient
data structure for sparse matrices, whose storage requirement must
be linear in the number of nonzeros. In the Crout update, we need
to access both rows and columns of $\vec{P}^{T}\hat{\vec{A}}\vec{Q}$,
$\vec{L}$, and $\vec{U}$, and it is critical that the access time
is constant per nonzero in an amortized sense. The standard storage
formats, such as AIJ, CCS (Compressed Column Storage), or CRS (Compressed
Row Storage), are insufficient for this purpose. In addition, we must
facilitate row interchanges in $\vec{L}$ and column interchanges
in $\vec{U}$ efficiently, while taking into account the cache performance
of the Crout update as much as possible. 

To this end, we introduce a data structure that augments the CCS or
CRS formats, which we refer to as \textit{Augmented Compressed Column
Storage }\textit{\emph{(}}AugCCS) and \textit{Augmented Compressed
Row Storage }\textit{\emph{(}}AugCRS), respectively. Like CCS and
CRS, our data structure is array based so that it can be easily implemented
in any programming language, such as MATLAB, FORTRAN, C/C++, etc.
The two versions are interchangeable in functionality, but they may
deliver different cache performance depending on the locality of the
algorithm. We will use AugCCS\textsf{ }and AugCRS\textsf{ }for $\vec{L}$
and $\vec{U}$, respectively, because PS-MILU primarily accesses $\vec{L}$
in columns and primarily accesses $\vec{U}$ in rows. In the following,
we shall describe AugCCS for $\vec{L}$.

Our primary design goals of AugCCS include building $\vec{L}$ incrementally
in columns in the Crout update, efficient sequential access in both
rows and columns, efficient search of row indices, and efficient implementation
of row interchanges. Recall that the standard CCS format has the following
collection of arrays:
\begin{itemize}
\item \textsf{row\_ind:}\textsf{\textbf{ }}An integer array of size equal
to the total number of nonzeros, storing the row indices of the nonzeros
in each column;
\item \textsf{col\_start: }An integer array of size $n+1$, storing the
start index of each column in \textsf{row\_ind}, where \textsf{col\_start{[}k+1{]}-1}
is the end index for the $k$th column;
\item \textsf{val: }A floating-point array of the same size as \textsf{row\_ind},
storing the nonzero values in each column.
\end{itemize}
In CCS, \textsf{row\_ind }and \textsf{val} have the same memory layout
and are accessed based on the same indexing. Because the number of
nonzeros in $\vec{\ell}_{k}$ is not known until the end of step $k$
of Crout update, we cannot determine \textsf{col\_start} \emph{a priori},
so \textsf{col\_start{[}k+1{]}} must be updated incrementally at the
end of the $k$th step. However, for efficiency, we can estimate the
total number of nonzeros in $\vec{L}$ and reserve storage for \textsf{row\_ind}
and \textsf{val} \emph{a priori}, because we control the number of
the nonzeros in $\vec{\ell}_{k}$ to be within a constant factor of
that in its corresponding column in $\vec{P}^{T}\hat{\vec{A}}\vec{P}$.

In CCS, the sorting of the indices in \textsf{row\_ind }is optional.
In AugCCS, we require the indices to be sorted in ascending order.
The sorting may be counterintuitive, because it\textsf{ }would incur
additional cost when appending a column $\vec{\ell}_{k}$ and when
performing row interchanges. Sorting entries in $\vec{\ell}_{k}$
has a lower time complexity compared to computing $\vec{\ell}_{k}$.
In terms of row interchanges, if rows $k$ and $i$ in $\vec{L}$
are interchanged at step $k$ of the Crout update, where $i>k$, we
need to update the sorted indices in the $j$th column in \textsf{row\_ind
}and in \textsf{val} if $\ell_{kj}$ or $\ell_{ij}$ is nonzero. The
total amount of data movement is linear in the number of nonzeros
within these columns, which is asymptotically optimal for pivoting
using a sparse storage format, as we will show in Section~\ref{subsec:Time-Complexity}.
With sorted lists, we can optimize other steps of the algorithm.

A main limitation of CCS is that it does not allow constant-time access
of nonzeros of $\vec{L}$ in row major, which is needed when updating
$\vec{u}_{k}$ using the nonzeros in $\vec{\ell}_{k}^{T}$. In \cite{li2003crout},
the authors dynamically update a linked list of nonzeros in $\vec{\ell}_{k+1}^{T}$
at step $k$ for Crout update without pivoting, but their procedure
would require $k$ comparisons if pivoting is performed, so the overall
time complexity would be quadratic. Another intuitive idea is to store
$\vec{L}$ in both CCS and CRS, but the number of nonzeros in $\vec{\ell}_{k}^{T}$
cannot be bounded \emph{a priori}, unless we limit the number of nonzeros
in each row, in addition to each column, of $\vec{L}$. This can lead
to more droppings and compromise the robustness of the preconditioner.

To overcome the preceding difficulties, we augment the CCS structure
with a linked-list version of CRS, implemented with the following
arrays in addition to \textsf{row\_ind}, \textsf{col\_start}, and\textsf{
val} in the base CCS implementation:
\begin{itemize}
\item \textsf{col\_ind:}\textsf{\textbf{ }}An integer array of size equal
to the maximum number of nonzeros, storing the column index of each
nonzero;
\item \textsf{row\_start:} An integer array of size $n$, storing the start
index of each row in \textsf{col\_ind};
\item \textsf{row\_next: }An integer array of the same size as \textsf{col\_ind},
storing the index in \textsf{col\_ind} for the next nonzero in the
same row;
\item \textsf{row\_end: }An integer array of size $n$, storing the last
index of each row in \textsf{col\_ind};
\item \textsf{val\_pos:} An integer array of the size as \textsf{col\_ind},
storing the index of the nonzeros in \textsf{val} in the base CCS
storage.
\end{itemize}
This linked-list version of CRS above allows easy expansion for each
row after computing $\vec{\ell}_{k}$. When exchanging rows $k$ and
$i$ in $\vec{L}$, besides updating \textsf{row\_ind }and \textsf{val},
we must also update their corresponding indices in \textsf{val\_pos},
and swap \textsf{row\_start{[}k{]} }and \textsf{row\_end{[}k{]} }with
\textsf{row\_start{[}i{]}} and \textsf{row\_end{[}i{]}}, correspondingly.
Updating \textsf{val\_pos }requires only linear time in the total
number of nonzeros in $\vec{\ell}_{k}^{T}\cup\vec{\ell}_{i}^{T}$.
Note that one can further extend AugCCS by replacing the array-based
CCS with its corresponding linked-list version, by adding \textsf{col\_next
}and \textsf{col\_end}; similarly for AugCRS. However, this is not
necessary for implementing PS-MILU.

\subsection{Algorithm Details and Pseudocode\label{subsec:Algorithm-Details}}

We now describe the implementation of the overall algorithm of PS-MILU,
focusing on the function \textsf{psmilu\_factor} in Section~\ref{subsec:Overall-Factorization-Algorithm}
for constructing the multilevel factorization. The output of \textsf{psmilu\_factor
}is a list of structure \textsf{Prec}, of which each entry corresponds
to a level and encapsulates the size information, $\vec{L}_{B}$,
$\vec{D}_{B}$, $\vec{U}_{B}$, $\vec{E}$, $\vec{F}$, $\vec{D}_{r}$,
$\vec{D}_{c}$, $\vec{P}$, and $\vec{Q}^{T}$ for that level. For
the last level, \textsf{Prec} stores the dense complete LU factorization
with pivoting of the Schur complement in \textsf{SC\_plu}, which is
of size 0 if $\vec{S}_{C}$ is empty; for the other levels, \textsf{SC\_plu}
is null. Figure\ \ref{fig:Desp of prec} demonstrates an implementation
of \textsf{Prec} in C99, but it can be easily adapted to any other
programming language. The output of \textsf{psmilu\_factor }will also
be the input to \textsf{psmilu\_solve}, which we will describe in
Section~\ref{subsec:Multilevel-Triangular-Solve}.

\begin{figure}
\begin{verbatim}
struct Prec {
  int     m, n;    // sizes of the leading block B and of A
  CCS     L_B;     // unit lower triangular factor of B
  double *d_B;     // diagonal entries in D_B
  CRS     U_B;     // unit upper triangular factor of B
  CRS     E, F;    // scaled and permuted blocks E and F

  double *s, *t;   // diagonal entries of scaling factors D_r & D_c
  int    *p;       // row permutation vectors of B
  int    *q_inv;   // inverse column permutation vectors of B

  PLU    *SC_plu;  // dense LU with pivoting of S_C in last level
};
\end{verbatim}

\caption{\label{fig:Desp of prec}Definition of data structure \textsf{Prec}
for each level of PS-MILU.}
\end{figure}

\subsubsection{Overall PS-MILU Algorithm\label{subsec:Overall-Factorization-Algorithm}}

We now present the pseudocode of PS-MILU. We refer to the top-level
function as \textsf{psmilu\_factor}, as shown in Algorithm\ \ref{Alg:psmilu_factor},
this function computes the preprocessing, ILU factorization with diagonal
pivoting, and the Schur complement $\vec{S}_{C}$, and recursively
factorizes $\vec{S}_{C}$. This function takes a matrix $\vec{A}$
of size $n\times n$ in CRS format as input. The first $m_{0}\times m_{0}$
leading block of $\vec{A}$ is assumed to be symmetric, where $m_{0}=0$
for a fully nonsymmetric matrix, and $m_{0}=n$ for a fully symmetric
matrix. In general, $m_{0}=0$ when \textsf{psmilu\_factor }is called
recursively on level 2 and onward. The PS-MILU has ten control parameters,
based on the theoretical analysis in Section~\ref{sec:PS-MILU}.
We encapsulate these parameters in a structure \textsf{Options}, as
shown in Figure\ \ref{fig:Desp of options} using C99 syntax. The
comments give the default values in our current implementation. We
also include the size of the input matrix in \textsf{Options}, since
it is needed at all levels. For simplicity, we use the same thresholds
for $\vec{L}$ and $\vec{U}$, although it is advantageous for $\tau_{L}<\tau_{U}$
and $\alpha_{L}<\alpha_{U}$ if it is known \emph{a priori} that the
factorization will be used as a right-preconditioner.

\begin{algorithm}
\caption{\textbf{\noun{\label{Alg:psmilu_factor}}} \textbf{psmilu\_factor}$(\vec{A},m_{0},\text{level},\text{options})$}

Computes PS-MILU of input matrix $\vec{A}$ and returns a list of
Prec

\textbf{input}: $\vec{A}$: original matrix in CRS format

\hspace{1.1cm} $m_{0}$: size of leading symmetric block

\hspace{1.1cm} options: $\tau_{L}$, $\tau_{U}$, $\tau_{\kappa}$,
$\tau_{d}$, $\alpha_{L}$, $\alpha_{U}$, $\alpha_{d}$, $\rho$,
$c_{d}$, $c_{h}$, $N$

\textbf{output}: precs: a list of Prec instances

\begin{algorithmic}[1]

\STATE $m\leftarrow m_{0}$; $n\leftarrow$size of $\vec{A}$

\STATE\textbf{ if} $m>0$ \textbf{then }\hfill \{partially symmetric\}

\STATE\hspace{0.4cm}\label{line:MC64} obtain $\vec{p},\vec{q},\vec{s},\vec{t}$
from symmetric MC64 on $\vec{A}_{1:m,1:m}$ and decrease $m$ if needed

\STATE\hspace{0.4cm} update $\vec{p},\vec{q}$ by symmetric AMD
on $\vec{A}[\vec{p},\vec{q}]$

\STATE\textbf{ else} \hfill \{fully nonsymmetric\}

\STATE\hspace{0.4cm} obtain $\vec{p},\vec{q},\vec{s},\vec{t}$ from
nonsymmetric MC64 on $\vec{A}$

\STATE\hspace{0.4cm} update $\vec{p}$ and $\vec{q}$ by symmetric
AMD on $\vec{A}\left[\vec{p},\vec{q}\right]+\vec{A}^{T}\left[\vec{q},\vec{p}\right]$

\STATE \textbf{end if}

\STATE\label{line:scaling} $\text{\ensuremath{\vec{A}\leftarrow}diag}(\vec{s})\vec{A}\text{diag}(\vec{t})$
\hfill{}\{scale and convert $\vec{A}$ to AugCRS\}

\STATE sym$\leftarrow$$m>0$ \textbf{and} level is 1

\STATE {[}$\vec{L}_{B}$, $\vec{d}_{B}$, $\vec{U}_{B}$, $\vec{L}_{E}$,
$\vec{U}_{F}$, $\vec{p}$, $\vec{q}$, $m${]} $\leftarrow$ \textbf{iludp\_factor}($\vec{A}$,
$\vec{p}$, $\vec{q}$, $m$, sym, options)

\STATE $\vec{E}\leftarrow\vec{A}[\vec{p}_{m+1:n},\vec{q}_{1:m}]$;
$\vec{F}\leftarrow\vec{A}[\vec{p}_{1:m},\vec{q}_{m+1:n}]$

\STATE save $m$, $n$, $\vec{L}_{B}$, $\vec{D}_{B}$, $\vec{U}_{B}$,
$\vec{E}$, $\vec{F}$, $\vec{P}$, $\vec{Q}^{T}$, $\vec{s}$, and
$\vec{t}$ into precs{[}level{]}

\STATE\label{line:S-version}$\vec{S}_{C}\leftarrow\vec{A}[\vec{p}_{m+1:n},\vec{q}_{m+1:n}]-\vec{L}_{E}\text{diag}\left(\vec{d}_{B}\right)\vec{U}_{F}$
\hfill{}\{S-version in CRS\}

\STATE \textbf{if} $\text{nnz}(\vec{S}_{C})\geq\rho(n-m)^{2}$ \textbf{or}
$n-m\leq c_{d}\sqrt[3]{N}$\textbf{ then} \hfill{}\{$\vec{S}_{C}$
is dense or small\}

\STATE\hspace{0.4cm} \textbf{if} $n-m<c_{h}$ \textbf{then}

\STATE\hspace{0.8cm} $\vec{T}_{E}\leftarrow\vec{L}_{E}\vec{L}_{B}^{-1}\left(\vec{A}[\vec{p}_{1:m},\vec{q}_{1:m}]-\text{diag}(\vec{d}_{B})\right)$;
$\vec{T}_{F}\leftarrow\vec{U}_{B}^{-1}\vec{U}_{F}$\hfill{}\{CRS
and CCS\}

\STATE\hspace{0.8cm} \label{line:H-version}$\vec{S}_{C}\leftarrow\vec{S}_{C}+\vec{T}_{E}\vec{T}_{F}$
\hfill{}\{H-version in dense format\}

\STATE\hspace{0.4cm} \textbf{end if}

\STATE\hspace{0.4cm} save dense LU with pivoting of $\vec{S}_{C}$
into precs{[}level{]}

\STATE \textbf{else}

\STATE\hspace{0.4cm} \textbf{psmilu\_factor}($\vec{S}_{C}$, $0$,
level+1, options)

\STATE \textbf{end if}

\end{algorithmic}
\end{algorithm}

\begin{figure}
\begin{verbatim}
struct Options {
  double tau_L;      // inverse-based threshold for L       [0.01]
  double tau_U;      // inverse-based threshold for U       [0.01]
  double tau_d;      // threshold for inverse-diagonal        [10]
  double tau_kappa;  // inverse-norm threshold; default      [100]
  int    alpha_L;    // growth factor of nonzeros per col      [4]
  int    alpha_U;    // growth factor of nonzeros per row      [4]
  double rho;        // density threshold for dense LU      [0.25]
  int    c_d;        // size parameter for dense LU            [1]
  int    c_h;        // size parameter for H-version          [10]
  int    N;          // reference size of matrix         [size(A)]
};
\end{verbatim}

\caption{\label{fig:Desp of options} Definition of control options of PS-MILU
along with their default values.}
\end{figure}

We note a few details in \textsf{psmilu\_factor}. In line \ref{line:MC64},
we need to permute the dense rows and columns out of the leading block
before applying MC64 and permute the $2\times2$ diagonal blocks out
of the block after MC64, and then decrease $m$ accordingly. To access
the permuted matrix, we use $\vec{A}[\vec{p},\vec{q}]$ to denote
$\vec{P}^{T}\vec{A}\vec{Q}$. In line \ref{line:scaling} we scale
the matrix $\vec{A}$ using the row and column scaling vectors from
MC64 and convert $\vec{A}$ into AugCRS format. The scaling is for
clarity of the pseudocode; the code can be adapted to scale $\vec{A}$
when accessing $\vec{A}[\vec{p},\vec{q}]$ in that $\vec{P}^{T}\vec{D}_{r}\vec{A}\vec{D}_{q}\vec{Q}=\text{diag}\left(\vec{s}[\vec{p}]\right)\vec{A}[\vec{p},\vec{q}]\text{diag}\left(\vec{t}[\vec{q}]\right)$.

\begin{algorithm}
\caption{\textbf{\noun{\label{alg:iludp_factor}}} \textbf{iludp\_factor}($\vec{A}$,
$\vec{p}$, $\vec{q}$, $m$, sym, options) }

Compute ILDU (or symmetric ILDL$^{\text{T}}$) with diagonal pivoting

\textbf{input}: $\vec{A}$: input scaled matrix at each level in AugCRS
format

\hspace{1.1cm} $\vec{p},\vec{q}$: row and column permutation vectors
of $\vec{A}$

\hspace{1.1cm} $m$: size of current leading block

\hspace{1.1cm} $\text{options}$: $\tau_{L}$, $\tau_{U}$, $\tau_{\kappa}$,
$\tau_{d}$, $\alpha_{L}$, $\alpha_{U}$

\textbf{output}: $\vec{L}_{B}$, $\vec{d}_{B}$, $\vec{U}_{B}$, $\vec{L}_{E}$,
$\vec{U}_{F}$, $\vec{p}$, $\vec{q}$, $m$

\begin{algorithmic}[1]

\STATE$\vec{d}[\vec{p}]\leftarrow\text{diag}(\vec{A}[\vec{p},\vec{q}])$;
reserve spaces for $\vec{L}$ (AugCCS ), $\vec{U}$ (AugCRS), $\hat{\vec{\ell}}$,
and $\hat{\vec{u}}$

\STATE\textbf{for $k=1,2,\dots m$} \hfill{}\{note that $m$ may
decrease within loop\}

\STATE\hspace{0.4cm} pivot$\leftarrow\vert1/d_{k}\vert>\tau_{d}$

\STATE\hspace{0.4cm}\textbf{ while true}

\STATE\hspace{0.8cm} \textbf{if} pivot\textbf{ then}

\STATE\hspace{1.2cm}\label{line:pivot} $m\leftarrow m-1$ \textbf{while}
$\vert1/d_{m}\vert>\tau_{d}$ and $m<k$

\STATE\hspace{1.2cm} \textbf{break if} $k=m$ \hfill{}\{jump to
line \ref{line:last}\}

\STATE\hspace{1.2cm} swap rows $k$ and $m$ in $\vec{L}$; swap
columns $k$ and $m$ in $\vec{U}$; update $\vec{p}$ and $\vec{q}$

\STATE\hspace{1.2cm} $m\leftarrow m-1$

\STATE\hspace{0.8cm} \textbf{end if}

\STATE\hspace{0.8cm}\label{line:Cront-L} $\hat{\vec{\ell}}\leftarrow\vec{A}[\vec{p}_{k+1:n},q_{k}]-\vec{L}_{k+1:n,1:k-1}\vec{D}_{k-1}\vec{u}_{k}$

\STATE\hspace{0.8cm}\label{line:Crout-U} \textbf{$\hat{\vec{u}}^{T}\leftarrow\vec{A}[p_{k},\vec{q}_{k+1:n}]-\vec{\ell}_{k}^{T}\vec{D}_{k-1}\vec{U}_{1:k-1:k+1:n}$}\hfill{}\{$\hat{\vec{u}}_{k+1:m}=\hat{\vec{\ell}}_{k+1:m}$
if sym\}

\STATE\hspace{0.8cm} estimate $\tilde{\kappa}_{L,k}$ and $\tilde{\kappa}_{U,k}$
by (\ref{eq:bound inv L}) and (\ref{eq:bound inv U}) \hfill{}\{$\tilde{\kappa}_{U,k}=\tilde{\kappa}_{L,k}$
if sym\}

\STATE\hspace{0.8cm} pivot$\leftarrow$$\tilde{\kappa}_{L,k}>\tau_{\kappa}$
\textbf{or} $\tilde{\kappa}_{U,k}>\tau_{\kappa}$

\STATE\hspace{0.8cm} \textbf{continue if} pivot\hfill{}\{jump to
line \ref{line:pivot}\}

\STATE\hspace{0.8cm} $d[p_{i}]\leftarrow d[p_{i}]-d[p_{k}]\hat{\ell}{}_{i}\hat{u}_{i}$
\textbf{for} $k<i\leq m\text{ \textbf{with} }\hat{\ell}{}_{i}\neq0\text{ \textbf{and }}\hat{u}_{i}\neq0$

\STATE\hspace{0.8cm} drop $\hat{\ell}_{i}$ \textbf{if} $\left|\hat{\ell}_{i}\right|\tilde{\kappa}_{L,k}\leq\tau_{L}$
\textbf{for} $k<i\leq n\text{ \textbf{with} }\hat{\ell}{}_{i}\neq0$

\STATE\hspace{0.8cm} drop $\hat{u}_{j}$ \textbf{if} $\left|\hat{u}_{j}\right|\tilde{\kappa}_{U,k}\leq\tau_{U}$
\textbf{for} $k<j\leq n\text{ \textbf{with} }\hat{u}{}_{j}\neq0$

\STATE\hspace{0.8cm} $n_{L}\leftarrow\alpha_{L}\text{ nnz}(\vec{A}[:,q_{k}])$;
$n_{U}\leftarrow\alpha_{U}\text{ nnz}(\vec{A}[p_{k},:])$

\STATE\hspace{0.8cm}\label{line:sort-L}\textbf{ }find largest $n_{L}$
nonzero entries in $\hat{\vec{\ell}}$ and append sorted list to $\vec{L}$

\STATE\hspace{0.8cm}\textbf{ }find largest $n_{U}$ nonzero entries
in $\hat{\vec{u}}^{T}$ and append sorted list to $\vec{U}$

\STATE\hspace{0.8cm}\textbf{ break}

\STATE\hspace{0.4cm} \textbf{end while}

\STATE\textbf{end for}

\STATE\label{line:last}$\vec{L}_{B}\leftarrow\vec{L}_{1:m,:}$;
$\vec{d}_{B}\leftarrow\vec{d}[\vec{p}_{1:m}]$; $\vec{U}_{B}\leftarrow\vec{U}_{:,1:m}$;
$\vec{L}_{E}\leftarrow\vec{L}_{m+1:n,:}$; $\vec{U}_{F}\leftarrow\vec{U}_{:,m+1:n}$

\end{algorithmic}
\end{algorithm}

The core of \textsf{psmilu\_factor} is \textsf{iludp\_factor}, which
is outlined in Algorithm~\ref{alg:iludp_factor}. The function \textsf{iludp\_factor
}takes a scaled matrix in AugCRS format, along with the permutation
vectors $\vec{p}$ and $\vec{q}$, the size of the leading block $m$,
and additional options. The function will perform ILDL$^{\text{T}}$
on the leading block if the argument \textsf{sym} is true. We note
a few of its implementation details for $\vec{L}$; the same apply
for $\vec{U}$ symmetrically. First, in line \ref{line:Cront-L},
we need to maintain the starting row indices of the nonzeros in each
column of $\vec{L}_{k:m,1:k-1}$. To this end, we maintain an integer
array \textsf{L\_start} of size $n$, whose $j$th entry contain the
index into \textsf{row\_ind} for the first nonzero in $\vec{L}_{k:n,j}$
for $j<k$. Updating these indices only requires $\mathcal{O}\left(\text{nnz}\left(\vec{\ell}_{k}^{T}\right)\right)$
operations at the $k$th step. Second, in line \ref{line:Cront-L},
we store $\hat{\vec{\ell}}$ as a dense array of size $m$, and build
the set of row indices of its nonzeros. The thresholding procedure
can then be computed in $\mathcal{O}\left(\text{nnz}\left(\hat{\vec{\ell}}\right)\right)$
operations, and sorting the nonzeros in line~\ref{line:sort-L} would
then take $\mathcal{O}\left(\text{nnz}\left(\vec{\ell}_{k}\right)\log\left(\text{nnz}\left(\vec{\ell}_{k}\right)\right)\right)$
operations after dropping. This is illustrated in the right panel
of Figure~\ref{fig:Illustration-update}. Third, for row interchanges
in $\vec{L}$, we also take advantage of \textsf{L\_start} to locate
the nonzeros in $\vec{L}_{k:m,1:k-1}$, and then update the AugCCS
format as described in Section~\ref{subsec:Data-Structure}. Finally,
for efficiency, all the storage is preallocated outside the loops.

After obtaining the ILU factorization with diagonal pivoting, we compute
the Schur complement and then factorize it. We first compute the S-version
in line \ref{line:S-version} of Algorithm~\ref{Alg:psmilu_factor}.
Like $\vec{A}$, $\vec{S}_{C}$ is first stored in CRS format, which
we compute similarly to the Crout update of $\vec{U}$ without dropping
and pivoting. We factorize the Schur complement using \textsf{psmilu\_factor}
recursively or using a dense complete factorization based on its size.
If a dense factorization is used, we update the S-version to the H-version
if the size of $\vec{S}_{C}$ is smaller than a constant.

\subsubsection{\label{subsec:Multilevel-Triangular-Solve}Multilevel Triangular
Solves}

From \textsf{psmilu\_factor}, we obtain a list of \textsf{Prec} instances,
which define a multilevel preconditioner. Let $\hat{\vec{A}}=\vec{P}^{T}\vec{D}_{r}\vec{A}\vec{D}_{c}\vec{Q}$,
where $\vec{P}$, $\vec{Q}$, $\vec{D}_{r}$, and $\vec{D}_{c}$ are
the permutation and scaling matrices from \textsf{psmilu\_factor}.
Similar to (\ref{eq:multilevel-inverse}), we can define a preconditioner
$\vec{M}$ corresponding to $\hat{\vec{A}}$ as

\begin{equation}
\vec{M}^{-1}=\begin{bmatrix}\tilde{\vec{B}} & \vec{0}\\
\vec{0} & \vec{0}
\end{bmatrix}+\begin{bmatrix}-\tilde{\vec{B}}^{-1}\vec{F}\\
\vec{I}
\end{bmatrix}\tilde{\vec{S}}_{C}^{-1}\begin{bmatrix}-\vec{E}\tilde{\vec{B}}^{-1} & \vec{I}\end{bmatrix},\label{eq:precond-permuted-scaled}
\end{equation}
where $\tilde{\vec{B}}=\vec{L}_{B}\vec{D}_{B}\vec{U}_{B}$, $\vec{E}$
and $\vec{F}$ are scaled and permuted blocks in $\hat{\vec{A}}$,
and $\tilde{\vec{S}}_{C}^{-1}$ is the multilevel preconditioner for
the Schur complement. Since $\tilde{\vec{B}}^{-1}$ is computed using
triangular solves and $\tilde{\vec{S}}_{C}^{-1}$ uses triangular
solves recursively, we refer to the procedure of computing $\vec{M}^{-1}\vec{b}$
as \emph{multilevel triangular solve }for $\vec{b}\in\mathbb{R}^{n}$. 

Although (\ref{eq:multilevel-inverse}) and (\ref{eq:precond-permuted-scaled})
are convenient mathematically, for user-friendliness the permutation
and scaling matrices should be transparent to the user. To this end,
we define a preconditioner, denoted as $\tilde{\vec{M}}$, with respect
to $\vec{A}$ instead of $\hat{\vec{A}}$. Let
\begin{equation}
\tilde{\vec{M}}^{-1}=\vec{D_{c}}\vec{Q}\vec{M}^{-1}\vec{P}^{T}\vec{D_{r}}.\label{eq:multilevel-preconditioner}
\end{equation}
Algorithm\ \ref{alg:Computing solve_milu} outlines the function
\textsf{psmilu\_solve} for evaluate $\vec{y}=\tilde{\vec{M}}^{-1}\vec{b}$.
Its input is a list of Precs instances, the current level, and a vector
$\vec{b}$. It is a recursive function, where $\tilde{\vec{S}}_{C}^{-1}$
is either computed recursively or solved using the dense factorization.

\begin{algorithm}
\caption{\label{alg:Computing solve_milu}\textbf{psmilu\_solve}(precs, level,
\textbf{$\vec{b}$})}

Computing $\vec{y}=\tilde{\vec{M}}^{-1}\vec{b}$, where $\tilde{\vec{M}}^{-1}$
is defined in (\ref{eq:multilevel-preconditioner})

\textbf{input}: precs: list of Prec instances

\hspace{1.1cm} level: current level

\hspace{1.1cm} $\vec{b}$: arbitrary vector

\textbf{output: $\vec{y}=\tilde{\vec{M}}^{-1}\vec{b}$ }

\textbf{local}: $m$, $n$, $\vec{L}_{B}$, $\vec{D}_{B}$, $\vec{U}_{B}$,
$\vec{E}$, $\vec{F}$, $\vec{p}$, $\vec{q}_{\text{inv}}$, $\vec{s}$,
$\vec{t}$, $\vec{S}_{C}^{-1}$ as aliases of those in precs{[}level{]}

\begin{algorithmic}[1]

\STATE $\vec{b}\leftarrow\vec{s}[\vec{p}].*\vec{b}[\vec{p}]$\textbf{
}\hfill\{element-wise multiplication\}

\STATE \textbf{$\vec{y}_{1:m}\leftarrow\vec{U}_{B}^{-1}\vec{D}_{B}^{-1}\vec{L}_{B}^{-1}\vec{b}_{1:m}$}

\STATE \textbf{$\vec{y}_{m+1:n}\leftarrow\vec{b}_{m+1:n}-\vec{E}\vec{y}_{1:m}$}

\STATE \textbf{if} $\vec{S}_{C}^{-1}$ is not null \textbf{and} $\text{size}\left(\vec{S}_{C}^{-1}\right)=n-m$
\textbf{then }\hfill\{dense factorization stored\}

\STATE\hspace{0.4cm}\textbf{ $\vec{y}_{m+1:n}\leftarrow\vec{S}_{C}^{-1}\vec{y}_{m+1:n}$}\hfill\{use
dense solves\}

\STATE \textbf{else}

\STATE\hspace{0.4cm}\textbf{ $\vec{y}_{m+1:n}\leftarrow$ psmilu\_solve}(precs,
level+1, $\vec{y}_{m+1:n}$)

\STATE \textbf{end if}

\STATE $\vec{y}_{1:m}\leftarrow\vec{y}_{1:m}-\vec{F}\vec{y}_{m+1:n}$

\STATE \textbf{$\vec{y}_{1:m}\leftarrow\vec{U}_{B}^{-1}\vec{D}_{B}^{-1}\vec{L}_{B}^{-1}\vec{y}_{1:m}$}

\STATE $\vec{y}\leftarrow\vec{t}.*\vec{y}[\vec{q}_{\text{inv}}]$\textbf{
}\hfill\{element-wise multiplication\}

\end{algorithmic}
\end{algorithm}

\subsection{Time Complexity Analysis\label{subsec:Time-Complexity}}

We now analyze the time complexity of PS-MILU. Typically, its overall
cost is dominated by the ILU factorization with diagonal pivoting
in the first level. Let \textsf{nnz() }denote a function for obtaining
the number of nonzeros in a matrix or vector. We first prove the following
lemma regarding the function \textsf{ilupd\_factor}.
\begin{lem}
\label{lem:time-complexity} Given a preprocessed matrix $\hat{\vec{A}}\in\mathbb{R}^{n\times n}$,
let $\vec{P}\in\mathbb{N}^{n\times n}$, $\vec{Q}\in\mathbb{N}^{n\times n}$,
$\vec{L}\in\mathbb{R}^{n\times m}$, $\vec{D}\in\mathbb{R}^{m\times m}$
and $\vec{U}\in\mathbb{R}^{m\times n}$ be the output of \textsf{ilupd\_factor}.
Suppose the number of diagonal pivoting for each diagonal entry is
bounded by a constant. The number of floating point operations in
Crout update is 
\begin{equation}
\mathcal{O}\left(\text{nnz}\left(\vec{L}+\vec{U}\right)\left(\max_{i\leq m}\left\{ \text{nnz}\left(\vec{a}_{i}^{T}\right)\right\} +\max_{j\leq m}\left\{ \text{nnz}\left(\vec{a}_{j}\right)\right\} \right)\right),\label{eq:bound-operations}
\end{equation}
so is the overall time complexity of \textsf{ilupd\_factor}, where
$\vec{a}_{i}^{T}$ and $\vec{a}_{j}$ denote the $i$th row and $j$th
column of $\vec{P}^{T}\hat{\vec{A}}\vec{Q}$, respectively.
\end{lem}
\begin{proof}
In \textsf{ilupd\_factor}, there are three main components: Crout
update, diagonal pivoting, and thresholding. First, consider the Crout
update without diagonal pivoting. At the $k$th step, the number of
floating-point operations of updating $\vec{\ell}_{k}$ is $2\sum_{\{i\mid i<k\text{ and }u_{ik}\neq0\}}\text{nnz}\left(\vec{L}_{k+1:n,i}\right)$,
which includes one multiplication and one addition per nonzero. The
total number of operations for computing $\vec{L}$ is 
\begin{align*}
N_{L} & =2\sum_{k=1}^{m}\sum_{\{i\mid i<k\text{ and }u_{ik}\neq0\}}\text{nnz}\left(\vec{L}_{k+1:n,i}\right)\\
 & \leq2\sum_{k=1}^{m}\sum_{\{i\mid i<k\text{ and }u_{ik}\neq0\}}\max_{j\leq m}\left\{ \text{nnz}\left(\vec{\ell}_{j}\right)\right\} \\
 & =2\text{nnz}\left(\vec{U}\right)\max_{j\leq m}\left\{ \text{nnz}\left(\vec{\ell}_{j}\right)\right\} \\
 & \leq2\text{nnz}\left(\vec{U}\right)\max_{j\leq m}\left\{ \text{nnz}\left(\vec{a}_{j}\right)\right\} ,
\end{align*}
where the last inequality is because we bound the number of nonzeros
in each column in $\vec{L}$ by a constant factor of its corresponding
column in $\vec{P}^{T}\hat{\vec{A}}\vec{Q}$. Similarly, the number
of floating operations in computing $\vec{U}$ is 
\[
N_{U}\leq2\text{nnz}\left(\vec{L}\right)\max_{i}\left\{ \text{nnz}\left(\vec{u}_{i}^{T}\right)\right\} \leq2\text{nnz}\left(\vec{L}\right)\max_{i\leq m}\left\{ \text{nnz}\left(\vec{a}_{i}^{T}\right)\right\} .
\]
Furthermore, the overall cost of updating $\vec{D}$ is 
\[
N_{D}\leq\sum_{k=1}^{m}\left(\text{nnz}\left(\vec{\ell}_{k}\right)+\text{nnz}\left(\vec{u}_{k}^{T}\right)\right)=\text{\text{nnz}\ensuremath{\left(\vec{L}+\vec{U}\right)}}.
\]
Hence, the total number of floating operations of Crout update is
\begin{align}
N_{L}+N_{U}+N_{D} & =\mathcal{O}\left(\text{nnz}\left(\vec{L}\right)\max_{i\leq m}\left\{ \text{nnz}\left(\vec{a}_{i}^{T}\right)\right\} +\text{nnz}\left(\vec{U}\right)\max_{j\leq m}\left\{ \text{nnz}\left(\vec{a}_{j}\right)\right\} \right)\label{eq:Crout-bound}\\
 & \leq\mathcal{O}\left(\text{nnz}\left(\vec{L}+\vec{U}\right)\left(\max_{i\leq m}\left\{ \text{nnz}\left(\vec{a}_{i}^{T}\right)\right\} +\max_{j\leq m}\left\{ \text{nnz}\left(\vec{a}_{j}\right)\right\} \right)\right).\label{eq:Crout-relaxed-bound}
\end{align}
With diagonal pivoting, under the assumption that each diagonal entry
is pivoted at most a constant number of times, the asymptotic rate
remains the same.

For the overall time complexity, let us consider the data movement
in diagonal pivoting next. Consider interchanging rows $\vec{\ell}_{k}^{T}$
and $\vec{\ell}_{i}^{T}$ in $\vec{L}_{k:n,1:k-1}$, where $i>k$.
If $\ell_{kj}\neq0$ or $\ell_{ij}\neq0$, maintaining \textsf{row\_ind}
sorted (along with \textsf{val}) for the nonzeros in $\vec{L}_{k:i,j}$
requires $\mathcal{O}(\text{nnz}(\vec{L}_{k:i,j}))$ operations, because
we keep track of the starting row indices of the nonzeros in $\vec{L}_{k:n,j}$
in CCS. Updating the augmented linked list requires linear time in
the number of nonzeros in $\vec{\ell}_{k}^{T}\cup\vec{\ell}_{i}^{T}$.
Hence, the total data movement of row interchange at step $k$ is
$\mathcal{O}\left(\text{nnz}\left(\bigcup_{\{j\mid\ell_{kj}\neq0\text{ or }\ell_{ij}\neq0\}}\vec{L}_{k:i,j}\right)\right)$,
and the overall cost of row interchanges in $\vec{L}$ is $\mathcal{O}(N_{L})$.
Similarly, the overall cost of column interchanges in $\vec{U}$ is
$\mathcal{O}(N_{U})$.

Finally, consider the cost of thresholding and sorting in $\vec{L}$.
At the $k$th step, its complexity is $\mathcal{O}\left(\text{nnz}\left(\hat{\vec{\ell}}_{k}\right)+\text{nnz}\left(\vec{\ell}_{k}\right)\log\left(\text{nnz}\left(\vec{\ell}_{k}\right)\right)\right)$,
where $\hat{\vec{\ell}}_{k}$ and $\vec{\ell}_{k}$ denotes the $k$th
column in $\vec{L}$ before and after dropping, respectively. Note
that $\sum_{k}\text{nnz}\left(\hat{\vec{\ell}}_{k}\right)=\mathcal{O}(N_{L})$.
Furthermore, 
\begin{align*}
\sum_{k}\text{nnz}\left(\vec{\ell}_{k}\right)\log\left(\text{nnz}\left(\vec{\ell}_{k}\right)\right) & \leq\text{nnz}\left(\vec{L}\right)\log\left(\max_{j\leq m}\left\{ \text{nnz}\left(\vec{\ell}_{j}\right)\right\} \right)\\
 & \leq\mathcal{O}\left(\text{nnz}\left(\vec{L}\right)\max_{j\leq m}\left\{ \text{nnz}\left(\vec{a}_{j}\right)\right\} \right).
\end{align*}
Similarly, thresholding in $\vec{U}$ takes $\mathcal{O}\left(N_{U}+\text{nnz}\left(\vec{U}\right)\max_{i\leq m}\left\{ \text{nnz}\left(\vec{a}_{i}^{T}\right)\right\} \right)$
operations. Hence, the total number of operations in \textsf{ilupd\_factor}
is bounded by (\ref{eq:bound-operations}).
\end{proof}
We note two details in Lemma~\ref{lem:time-complexity}. First, $\text{nnz}\left(\vec{a}_{i}^{T}\right)$
and $\text{nnz}\left(\vec{a}_{j}\right)$ only consider the first
$m$ rows and columns in $\vec{P}^{T}\hat{\vec{A}}\vec{Q}$. This
is important if there are a small number of dense rows or columns
in the input matrix, which should be permuted to the end to avoid
increasing the time complexity. Second, (\ref{eq:Crout-bound}) gives
a tighter bound for the floating-point operations in Crout update
if $\max_{j\leq m}\left\{ \text{nnz}\left(\vec{a}_{j}\right)\right\} $
and $\max_{i\leq m}\left\{ \text{nnz}\left(\vec{a}_{i}^{T}\right)\right\} $
have different complexity. In particular, if $\max_{j\leq m}\left\{ \text{nnz}\left(\vec{a}_{j}\right)\right\} =\mathcal{O}\left(\log\left(\max_{i\leq m}\left\{ \text{nnz}\left(\vec{a}_{i}^{T}\right)\right\} \right)\right)$
or vice versa, then sorting may become the most expensive component
of the algorithm. However, typically $\max_{j\leq m}\left\{ \text{nnz}\left(\vec{a}_{j}\right)\right\} =\mathcal{O}\left(\max_{i\leq m}\left\{ \text{nnz}\left(\vec{a}_{i}^{T}\right)\right\} \right)$
and vice versa, especially after matching and reordering. Hence, we
use the relaxed bound in (\ref{eq:Crout-relaxed-bound}) for Crout
update, and then the cost of sorting becomes negligible.

An important special case of Lemma~\ref{lem:time-complexity} is
when the nonzeros per row and per column in the input are bounded
by a constant. This is typically the case for linear systems arising
from PDE discretizations using finite difference or finite element
methods. In this case, we have the following proposition.
\begin{prop}
\label{prop:linear-complexity} If the input matrix $\vec{A}\in\mathbb{R}^{n\times n}$
has a constant number of nonzeros per row and per column, assuming
the number of diagonal pivoting per diagonal entry is bounded by a
constant, the total cost of \textsf{ilupd\_factor} is linear in $n$.
\end{prop}
\begin{proof}
Under the assumptions, $\max_{i\leq n}\left\{ \text{nnz}\left(\vec{a}_{i}^{T}\right)\right\} +\max_{j\leq n}\left\{ \text{nnz}\left(\vec{a}_{j}\right)\right\} =\mathcal{O}(1)$.
Furthermore, $\text{nnz}\left(\vec{L}+\vec{U}\right)=\text{nnz}\left(\vec{A}\right)=\mathcal{O}(n)$.
\end{proof}
For the overall cost of PS-MILU, we also need to take into account
the cost of preprocessing, including matching and reordering, as well
as factorizing the Schur complement. The preprocessing cost is typically
insignificant for practical problems, even for those with millions
of unknowns. However, their complexities may be quadratic or worse
in the number of unknowns \cite{Heggernes01CCM,olschowka1996new}.
Developing scalable matching and reordering method are challenging
problems in their own right, and it is beyond the scope of this paper.
In addition, since we do not drop nonzeros in the Schur complement,
its nonzeros may grow superlinearly in the worst case, but this rarely
happens with the S-version. With the T-version or the H-version, however,
$\vec{T}_{E}$ and $\vec{T}_{F}$ in line~\ref{line:H-version} of
Algorithm~\ref{Alg:psmilu_factor} are nearly dense, and computing
$\vec{T}_{E}$ and $\vec{T}_{F}$ would lead to superlinear cost unless
the size of the Schur complement is bounded by a constant. Hence,
we use the H-version only in the last level when the size of $\vec{S}_{C}$
is smaller than a constant $c_{h}$. Assuming the preprocessing cost
is negligible, we expect both \textsf{psmilu\_factor} and \textsf{psmilu\_solve}
to scale nearly linearly for the systems arising from PDE discretizations,
under the assumptions in Proposition~\ref{prop:linear-complexity}.
We will demonstrate the near linear scaling of \textsf{psmilu\_factor
}in Section~\ref{subsec:Asymptotic-Growth-with}.

\section{Numerical Study\label{sec:Numerical-Results}}

To verify the theoretical analysis and the algorithm in this paper,
we have developed a proof-of-concept implementation of PS-MILU in
MATLAB. In this section, we present some preliminary numerical results
to demonstrate the effectiveness of PS-MILU. To this end, we have
constructed a collection of linear systems from PDE discretizations
and solved them using restarted GMRES with PS-MILU as a right-preconditioner.
We also compare PS-MILU against ILUPACK and SuperLU to demonstrate
that PS-MILU scales better than these two of the best variants of
ILU.

\subsection{Benchmark problems}

In the literature, most experiments of ILU use small systems, such
as those from the Matrix Market \cite{boisvert1997matrix} or the
UF Sparse Matrix Collection \cite{davis2011university}. To verify
the robustness of PS-MILU, we have successfully tested it using some
of these problems, including those from \cite{bollhofer2011ilupack}
and \cite{li2005overview}. However, most of those systems have only
thousands to tens of thousands of unknowns, so they are not representative
of linear systems arising from modern PDE-based applications. In addition,
there are very few predominantly symmetric systems in those collections,
there is a lack of matrix series for scalability analysis, and most
systems do not have the right-hand-side vectors, which are important
for studying the convergence of KSP methods. For these reasons, we
built a collection of predominantly symmetric benchmark systems with
up to about half a million unknowns. Table~\ref{tab:test_matrices}
summarizes these matrices, including their origins, the numbers unknowns
and nonzeros, and the estimated condition numbers. For completeness,
we describe the numerical methods used in generating these systems
below.

\begin{table}
\caption{\label{tab:test_matrices}Summary of test matrices. }

\centering{}%
\begin{tabular}{c|c|c|c|c}
\hline 
\textbf{Matrix ID} & \textbf{Method} & \textbf{\#Unknowns} & \textbf{\#Nonzeros} & \textbf{Cond. No.}\tabularnewline
\hline 
\hline 
E2d1 & \multirow{7}{*}{FEM 2D} & 32,575 & 368,086 & 3.88e4\tabularnewline
\cline{1-1} \cline{3-5} 
E2d2 &  & 54,953 & 623,260 & 6.79e4\tabularnewline
\cline{1-1} \cline{3-5} 
E2d3 &  & 72,474 & 823,472 & 1.06e5\tabularnewline
\cline{1-1} \cline{3-5} 
E2d4 &  & 87,744 & 998,298 & 1.13e5\tabularnewline
\cline{1-1} \cline{3-5} 
E2d5 &  & 117,449 & 1,338,133 & 1.46e5\tabularnewline
\cline{1-1} \cline{3-5} 
E2d6 &  & 268,259 & 3,066,125 & 3.94e5\tabularnewline
\cline{1-1} \cline{3-5} 
E2d7 &  & 456,933 & 5,229,627 & 5.59e5\tabularnewline
\hline 
E3s1 & \multirow{9}{*}{FEM 2D} & 26,724 & 336,156 & 1.48e3\tabularnewline
\cline{1-1} \cline{3-5} 
E3s2 &  & 51,043 & 668,529 & 2.39e3\tabularnewline
\cline{1-1} \cline{3-5} 
E3s3 &  & 75,327 & 1,005,606 & 3.31e3\tabularnewline
\cline{1-1} \cline{3-5} 
E3s4 &  & 117,683 & 1,609,248 & 5.26e3\tabularnewline
\cline{1-1} \cline{3-5} 
E3s5 &  & 200,433 & 2,800,888 & 8.78e3\tabularnewline
\cline{1-1} \cline{3-5} 
E3s6 &  & 386,060  & 5,509,359 & 1.66e4\tabularnewline
\cline{1-1} \cline{3-5} 
E3t1 &  & 98,068 & 1,200,308 & 1.62e8\tabularnewline
\cline{1-1} \cline{3-5} 
E3t2 &  & 177,645 & 2,327,903 & 1.59e8\tabularnewline
\cline{1-1} \cline{3-5} 
E3t3 &  & 315,148 & 4,138,040 & 8.19e8\tabularnewline
\hline 
D2q1 & \multirow{2}{*}{FDM 2D} & 158,802 & 792,416 & 1.63e5\tabularnewline
\cline{1-1} \cline{3-5} 
D2q2 &  & 248,502 & 1,240,516 & 2.56e5\tabularnewline
\hline 
D3c1 & \multirow{2}{*}{FDM 3D} & 112,896 & 776,256 & 2.32e3\tabularnewline
\cline{1-1} \cline{3-5} 
D3c2 &  & 219,600 & 1,515,360 & 3.57e3\tabularnewline
\hline 
\end{tabular}
\end{table}

All of these linear systems were constructed from discretizing the
Poisson equation with Dirichlet and Neumann boundary conditions, i.e.,
\begin{align}
-\nabla^{2}u & =f\qquad\,\,\,\,\text{ in }\Omega,\label{eq:Poisson1}\\
u & =g\qquad\text{ on }\partial\Omega_{D},\label{eq:Poisson-2}\\
\frac{\partial u}{\partial\vec{n}} & =h\qquad\text{ on }\partial\Omega_{N},
\end{align}
where $\vec{n}$ denotes the outward normal to the domain. Different
domains, numerical methods, and boundary conditions can lead to linear
systems with different nonzero patterns and condition numbers. We
considered both finite elements (FEM) and finite differences (FDM)
in 2D and 3D. For FEM, we used the above equation with Dirichlet boundary
conditions, discretized using quadratic, six-node triangles over a
unit circle in 2D and using linear tetrahedra over a sphere and torus
in 3D. All the meshes were generated using Gmsh \cite{Geuzaine2009gmsh}.
Figure\ \ref{fig:Meshes} shows some example meshes at a very coarse
level. The torus meshes tend to have poor element qualities and hence
more ill-conditioned linear systems. We construct the right-hand side
and the boundary conditions by differentiating the analytic functions
\begin{align}
u_{1} & =2ye^{x+y}\qquad\,\,\,\,\qquad\,\,\,\,\qquad\,\,\,\qquad\qquad\text{ in 2D, }\label{eq:analytical 2D fem}\\
u_{2} & =-(4(x^{2}+y^{2}+z^{2})+6)e^{x^{2}+y^{2}+z^{2}}\,\,\,\,\text{ in 3D.}\label{eq:analytical 3D fem}
\end{align}
The Dirichlet boundary nodes were not eliminated from the systems,
and hence the matrices are predominantly symmetric. For FDM, we used
centered differences over a square in 2D and a cube in 3D, subject
to the Neumann boundary conditions on the top and Dirichlet boundary
conditions elsewhere. We construct the right-hand side and the boundary
conditions by differentiating the analytic functions 
\begin{align}
u_{1} & =e^{x+y}\qquad\text{in 2D,}\label{eq:analytical 2D fem-1}\\
u_{2} & =e^{x+y+z}\,\,\,\,\text{ in 3D.}\label{eq:analytical 3D fem-1}
\end{align}
All of these linear systems are predominantly symmetric, where the
nonsymmetric parts are only due to boundary nodes.

\begin{figure}
\centering{}\includegraphics[width=0.3\textwidth]{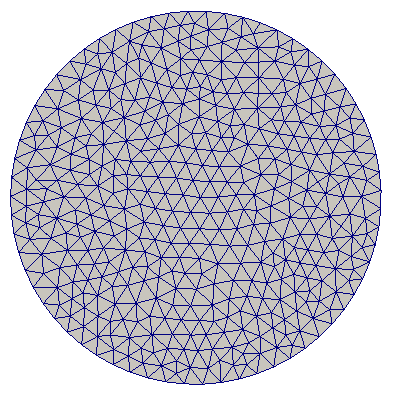}\includegraphics[width=0.39\textwidth]{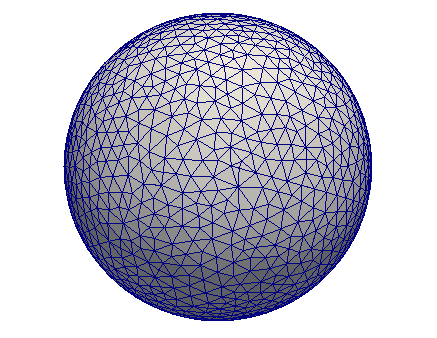}\includegraphics[width=0.3\textwidth]{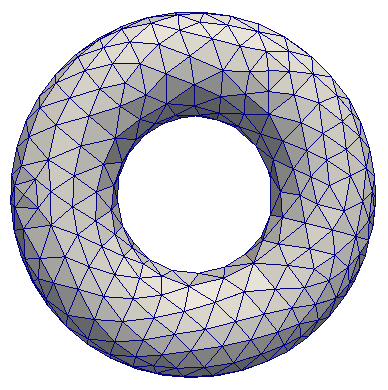}\caption{\label{fig:Meshes}Sample coarse meshes of the unstructured meshes
used for the 2D and 3D FEM.}
\end{figure}

In Table\ \ref{tab:test_matrices}, we identify each matrix by its
discretization type, spatial dimension, geometry, etc. In particular,
the first letter indicates the discretization methods, where \emph{E}
stands for \emph{FEM }and\emph{ D} for finite difference methods (\emph{FDM}).
It is then followed by the dimension (2 or 3) of the geometry. The
next lower-case letter indicates different types of domains, where
\textit{d} stands for \textit{disk}, \textit{s} for \textit{sphere},
\textit{t} for \textit{torus}, \textit{q} for \textit{square}, and
\textit{c} for \textit{cube}. If different mesh sizes of the same
problems were used, we append a digit to the matrix ID, where a larger
digit corresponds to a finer mesh. The condition numbers are in 1-norm,
estimated using MATLAB\textquoteright s \textsf{condest} function. 

\subsection{\label{subsec:Asymptotic-Growth-with}Verification of Linear Complexity}

For all of our test matrices, the number of nonzeros per row and per
column is bounded by a constant. This is typical for linear systems
from PDE discretization methods that use compact local stencils. Hence,
their computational cost should ideally be linear in the number of
unknowns. To verify the complexity analysis in Proposition~\ref{prop:linear-complexity},
we consider matrices E2d1--7 from FEM 2D, whose unknowns range between
approximately $37$K and $450$K, as well as E3s1--6 from FEM 3D,
whose unknowns range between approximately $26$K and $386$K unknowns.
Figure~\ref{fig:Scalability} shows the scalability plots of the
overall factorization time for FEM 2D (left) and 3D (right), respectively.
The \textit{x}-axis corresponds to the number of unknowns, and the
\textit{y}-axis corresponds to the normalized costs, both in logarithmic
scale. For PS-MILU, we show both the estimated floating-point operations
in level-1, along with the overall computational cost at all levels,
including the preprocessing steps. We used the default tolerances
as in Figure~\ref{fig:Desp of options}. As references, we also plot
the costs for multilevel ILU in ILUPACK v2.4 \cite{ilupack} and supernodal
ILU in SuperLU v5.2.1 \cite{li2005overview} for these systems, both
with their respective default parameters. All the tests were conducted
in serial on a single node of a cluster with two 2.5 GHz Intel Xeon
CPU E5-2680v3 processors and 64 GB of memory. For accurate timing,
both turbo and power saving modes were turned off for the processors,
and each node was dedicated to run one problem at a time. Since the
implementations may have a significant impact on the actual runtime,
we normalize each of these measures by dividing it by its corresponding
measure for the coarsest mesh.
\begin{flushleft}
\begin{figure}
\begin{minipage}[t]{0.45\textwidth}%
\begin{center}
\includegraphics[width=1\textwidth]{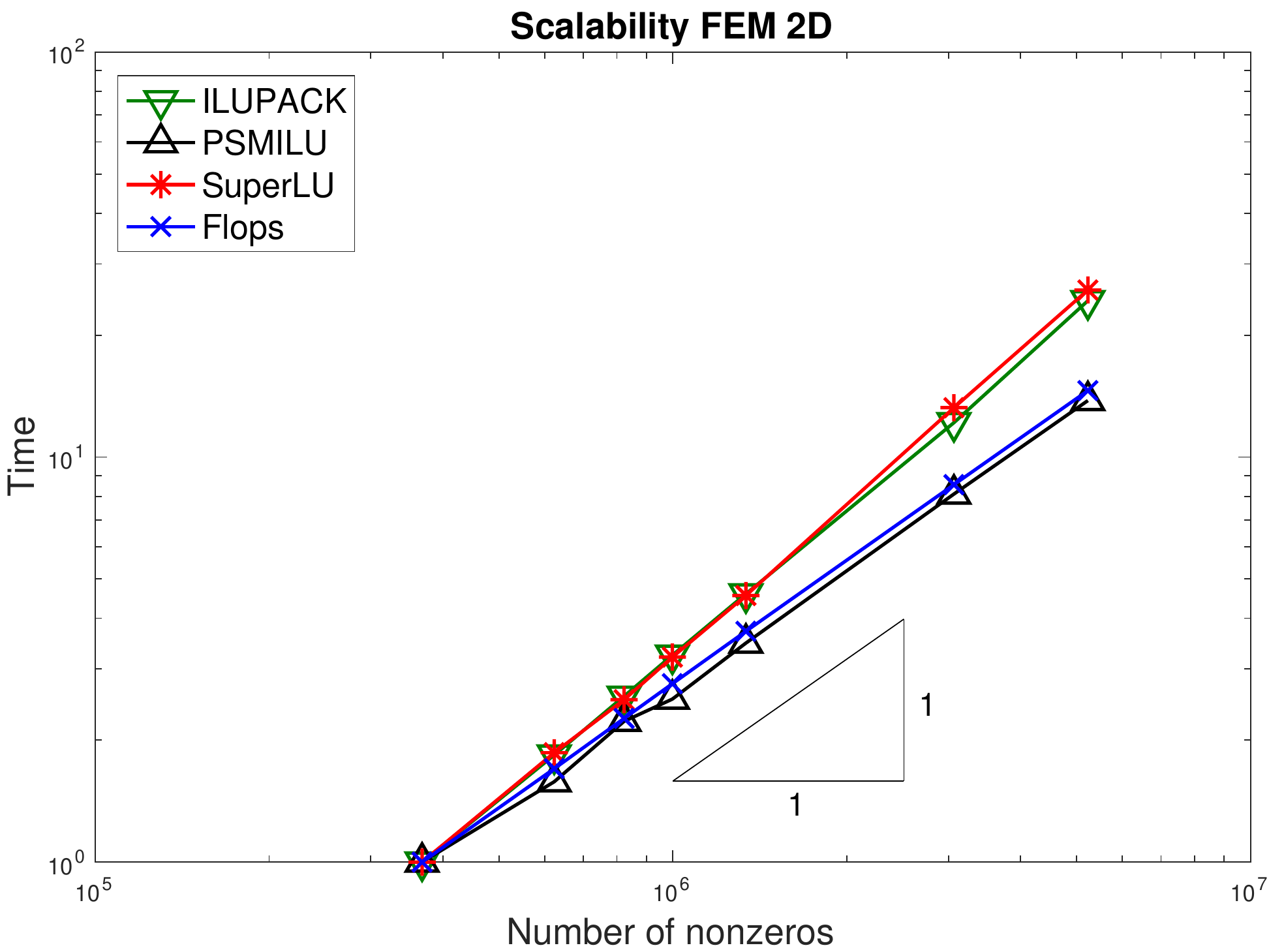}
\par\end{center}%
\end{minipage}\hfill{} %
\begin{minipage}[t]{0.45\textwidth}%
\begin{center}
\includegraphics[width=1\textwidth]{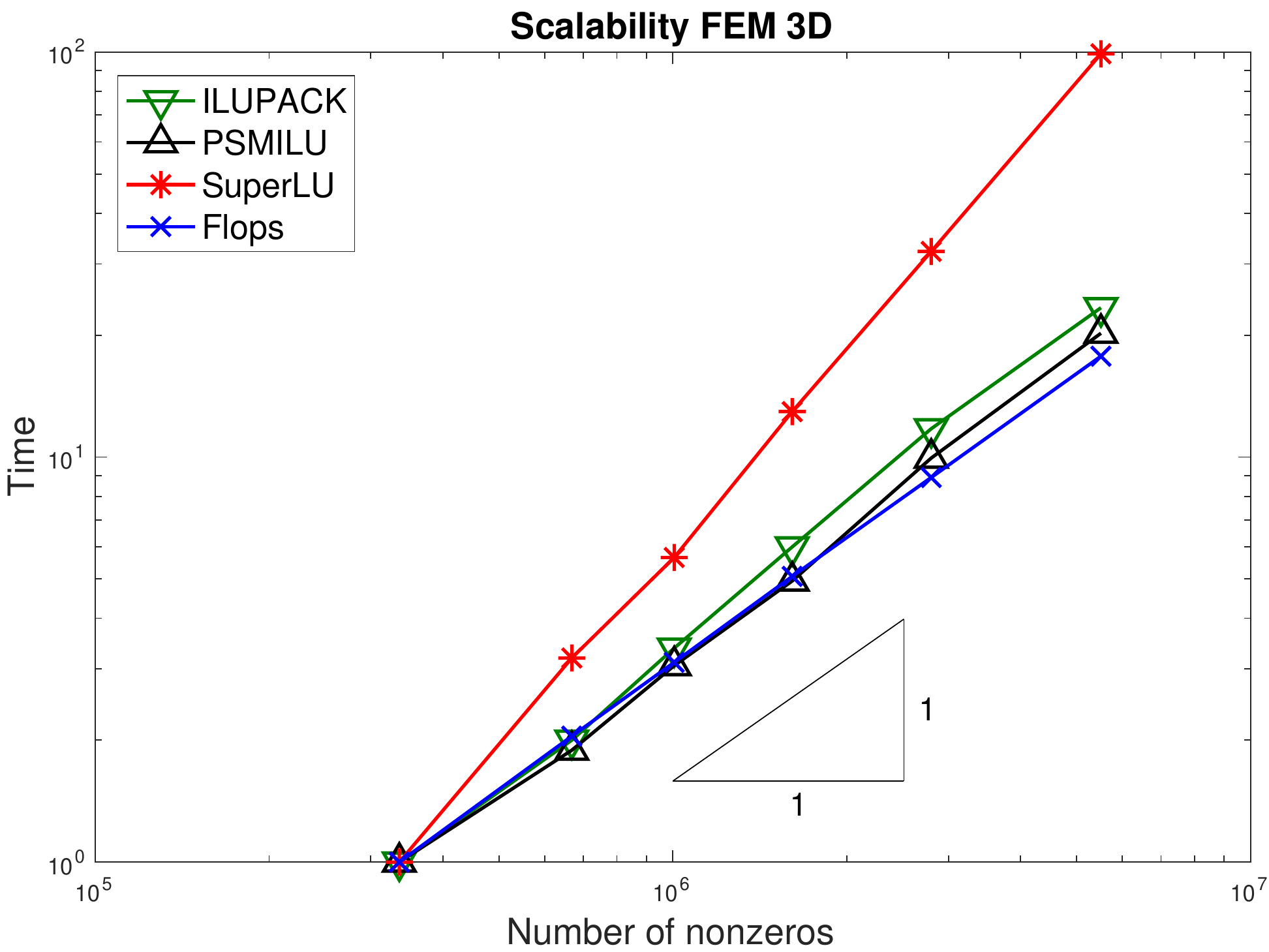}
\par\end{center}%
\end{minipage}\caption{\label{fig:Scalability} Asymptotic growth of factorization times
of E2d1--7 (left) and E3s1--6 with PS-MILU along with ILUPACK and
SuperLU, compared against theoretical flop counts.}
\end{figure}
\par\end{flushleft}

In Figure~\ref{fig:Scalability}, it can be seen that the total number
of floating-point operations in level-1 scales linearly, as predicted
by Proposition~\ref{prop:linear-complexity}. Since the updates in
level-1 dominate the overall algorithm, PS-MILU exhibited almost linear
scaling overall, despite nearly quadratic cost of the pre-processing
steps. In contrast, the supernodal ILU in SuperLU scales superlinearly,
and has much higher complexities in 3D than 2D. This is primarily
due to the use of column pivoting in SuperLU. On the other hand, ILUPACK
scales much worse than PS-MILU in 2D, at a rate comparable with SuperLU,
although it scales nearly as good as PS-MILU for 3D problems. This
is because the condition numbers of the linear systems from Poisson
equations grow at the rate of $\mathcal{O}(h^{-2})$, where $h$ denotes
the average edge length for quasi-uniform meshes. In 2D and 3D, the
condition numbers are approximately $\mathcal{O}(n)$ and $\mathcal{O}(n^{2/3})$,
respectively. Hence, with similar numbers of unknowns, 2D problems
tend to lead to more pivoting and more levels with multilevel ILU.
Because ILUPACK switches to use the T-version of the Schur complements
when the number of levels is large, it suffered from higher time complexities
for 2D problems than 3D problems. PS-MILU overcomes their scalability
issues by using diagonal pivoting instead of column pivoting and using
the H-version only for constant-size Schur complements.

\subsection{Speedup due to Predominant Symmetry}

In Table~\ref{tab:test_matrices}, due to the surface-to-volume ratio,
the symmetric portions of these systems are predominant. PS-MILU may
speed up the factorization by a factor of $2$ for the update of $\vec{L}_{B}$
and $\vec{U}_{B}$ by utilizing symmetry. However, since both $\vec{L}_{E}$
and $\vec{U}_{F}$ must be updated in PS-MILU, we cannot expect a
perfect twofold speedup. Table\ \ref{tab:PS-MILU versus MILU-speedup}
compares the speedups of PS-MILU versus nonsymmetric MILU at level
1 for 14 matrices. We show both the speedups of Crout update alone
and those of the overall factorization step. The Crout update achieved
a nearly twofold speedup for most cases, and the overall factorization
achieved a speedup between 1.4 and $1.6$.

\begin{table}
\caption{\label{tab:PS-MILU versus MILU-speedup} Speedups of PS-MILU versus
nonsymmetric MILU by taking advantage of predominant symmetry at level
1.}

\centering{}%
\begin{tabular}{>{\centering}m{2cm}|c|c}
\hline 
\multirow{1}{2cm}{\textbf{Matrix ID}} & \textbf{Update Speedup} & \textbf{Overall Speedup}\tabularnewline
\hline 
\hline 
E2d4 & 1.8 & 1.5\tabularnewline
\hline 
E2d5 & 1.9 & 1.5\tabularnewline
\hline 
E2d6 & 1.8 & 1.5\tabularnewline
\hline 
E2d7 & 1.8 & 1.5\tabularnewline
\hline 
E3s4 & 1.8 & 1.5\tabularnewline
\hline 
E3s5 & 1.9 & 1.6\tabularnewline
\hline 
E3s6 & 1.8 & 1.6\tabularnewline
\hline 
E3t1 & 1.7 & 1.5\tabularnewline
\hline 
E3t2 & 1.7 & 1.5\tabularnewline
\hline 
E3t3 & 1.7 & 1.5\tabularnewline
\hline 
D2q1 & 1.7 & 1.4\tabularnewline
\hline 
D2q2 & 1.7 & 1.4\tabularnewline
\hline 
D3c1 & 1.9 & 1.5\tabularnewline
\hline 
D3c2 & 1.9 & 1.5\tabularnewline
\hline 
\end{tabular}
\end{table}

\subsection{Effectiveness as Preconditioners}

Finally, we assess the effectiveness of PS-MILU as a preconditioner.
In particular, we compare PS-MILU versus using the fully nonsymmetric
MILU as a right-preconditioner of GMRES(30) for solving predominantly
symmetric systems. Table\ \ref{tab:PS-MILU versus MILU-fillin} shows
the ratios of the numbers of nonzeros in the output versus that in
the input (i.e., $\text{nnz}(\vec{L}+\vec{D}+\vec{U})/\text{nnz}(\vec{A})$),
the number of diagonal pivots, and the number of GMRES iterations
for both nonsymmetric MILU and PS-MILU. Our default thresholds in
the previous subsection worked for all of the test problems. In Table~\ref{tab:PS-MILU versus MILU-fillin},
we used the default drop-tolerances for most problems. For the well-conditioned
E3s1-6 and D3c1-2, there was no pivoting with the default parameters,
so we used $\tau_{k}=20$ to force some pivoting. It appeared that
nonsymmetric MILU is more sensitive to the thresholds than PS-MILU,
in that a smaller $\tau_{k}$ tends to lead to more pivoting for nonsymmetric
MILU than for PS-MILU. 
\begin{table}
\caption{\label{tab:PS-MILU versus MILU-fillin} Comparison of PS-MILU versus
nonsymmetric MILU as a right-preconditioner for GMRES(30). In the
GMRES columns, the numbers in each entry are the numbers of iterations
with relative convergence tolerances of $10^{-6}$ and $10^{-12}$,
respectively}

\centering{}%
\begin{longtable}{lcccccc}
\toprule 
\multirow{2}{*}{\textbf{Matrix}} & \multicolumn{3}{c}{\textbf{nonsymmetric MILU}} & \multicolumn{3}{c}{\textbf{PS-MILU }}\tabularnewline
\cmidrule{2-7} 
 & \textbf{ Ratio} & \textbf{\#Pivots} & \textbf{GMRES} & \textbf{ Ratio} & \textbf{\#Pivots} & \textbf{GMRES }\tabularnewline
\midrule
\midrule 
E2d4 & 2.00 & 430 & 59/128 & 2.01 & 401 & 59/123\tabularnewline
\midrule 
E2d5 & 2.01 & 613 & 63/157 & 2.07 & 555 & 63/158\tabularnewline
\midrule 
E2d6 & 2.01 & 1447 & 113/305 & 2.09 & 1335 & 111/299\tabularnewline
\midrule 
E2d7 & 2.00 & 2503 & 145/410 & 2.09 & 2302 & 144/404\tabularnewline
\midrule 
E3s4 & 2.26 & 476 & 16/37 & 2.33 & 450 & 16/38\tabularnewline
\midrule 
E3s5 & 2.28 & 1005 & 17/44 & 2.34 & 728 & 17/46\tabularnewline
\midrule 
E3s6 & 2.31 & 2213 & 19/55 & 2.38 & 1543 & 20/57\tabularnewline
\midrule 
E3t1 & 2.27 & 123 & 34/97 & 2.21 & 71 & \textbf{25/67}\tabularnewline
\midrule
E3t2 & 2.28 & 167 & 35/108 & 2.26 & 110 & \textbf{30/88}\tabularnewline
\midrule
E3t3 & 2.33 & 270 & 60/166 & 2.32 & 199 & \textbf{42/124}\tabularnewline
\midrule 
D2q1 & 3.58 & 468 & 72/203 & 3.81 & 345 & 72/204\tabularnewline
\midrule 
D2q2 & 3.58 & 600 & 109/282 & 3.77 & 623 & 112/295\tabularnewline
\midrule 
D3c1 & 4.41 & 652 & 21/41 & 4.56 & 411 & 21/41\tabularnewline
\midrule 
D3c2 & 4.45 & 1188 & 25/51 & 4.60 & 844 & 26/52\tabularnewline
\bottomrule
\end{longtable}
\end{table}

To assess the robustness of PS-MILU as a preconditioner, we use a
relative convergence tolerance of $10^{-12}$ for GEMRES(30). It is
well known that the restarted GEMRES tends to stagnate for large systems
without a good preconditioner. However, GMRES(30) with PS-MILU succeeded
for all of our test cases for such a small tolerance. Since some software
(such as MATLAB) use $10^{-6}$ as the default convergence tolerance
for GMRES, we report the numbers of iterations of GMRES(30) to achieve
both $10^{-6}$ and $10^{-12}$. It can be seen that PS-MILU and MILU
produced comparable numbers of nonzeros and GMRES iterations for most
cases. However, nonsymmetric MILU tends to introduce more pivots.
In addition, PS-MILU accelerated the convergence of GMRES better than
nonsymmetric MILU for E3t1--3, which are the most ill-conditioned
linear systems in our tests. Note that the number of iterations more
than doubled when squaring the convergence tolerance, indicating a
sub-linear convergence rate of GMRES(30) due to restarts. In addition,
the number of iterations still grew as the problem size increased.
We observed the same behavior with ILUPACK. Hence, in spite of the
linear complexity PS-MILU, one cannot expect the overall time complexity
to be linear when using PS-MILU as a preconditioner of a Krylov subspace
method. For a truly scalable solver, one would still need to use multigrid
methods.

\section{Conclusions and Future Work\label{sec:Conclusions and Future Work}}

In this paper, we proposed a multilevel incomplete LU factorization
technique, called \emph{PS-MILU}, as a preconditioner for Krylov subspace
methods. PS-MILU unifies the treatment of symmetric and nonsymmetric
linear systems, and it is robust due to the use of scaling, diagonal
pivoting, inverse-based thresholding, and systematic treatment of
the Schur complement. Its computational cost is nearly linear in the
number of unknowns for typical linear systems arising from PDE discretizations.
This is achieved by introducing augmented CCS and CRS data structures
and conducting careful complexity analysis of the algorithm. In addition,
we also introduced the concept of predominantly symmetric matrices.
We showed that PS-MILU can take advantage of this partial symmetry
to speed up the update operations by nearly a factor of two. We have
implemented the proposed algorithm in MATLAB and reported numerical
experimentation to demonstrate its robustness and linear scaling for
a collection of benchmark problems with up to half a million unknowns.
In addition, we compared PS-MILU against multilevel ILU in ILUPACK
and the supernodal ILU and analyzed the reasons of their poor scaling,
and explained how PS-MILU avoided those issues.

There are several limitations in this work. First, we primarily considered
linear systems from PDE discretizations. Although we have tested the
robustness of PS-MILU for smaller benchmark problems from other domains
in the literature, we have not assessed the scalability for large
systems from other applications, such as large KKT systems arising
from constrained optimizations. Second, we only considered MC64 matching
during preprocessing. There are other preprocessing techniques, such
as PQ-reordering \cite{saad2005multilevel}, which may be beneficial
for PS-MILU. Third, our current algorithm is only sequential, which
will ultimately limit the sizes of the problems that can be solved.
Finally, our proof-of-concept implementation is in MATLAB, which is
not the most efficient. We plan to address these issues in our future
research.

\section*{Acknowledgments}

Results were obtained using the LI-RED computer system at the Institute
for Advanced Computational Science of Stony Brook University, funded
by the Empire State Development grant NYS \#28451. We thank Dr. Matthias
Bollhöfer for sharing his ILUPACK code with us. We thank our colleagues
Yipeng Li and Qiao Chen for help with generating the FEM and FDM test
cases.

\bibliographystyle{abbrv}
\bibliography{refs/refs,refs/multigrid,refs/compkrylov_refs,refs/psmilu,refs/refs_18}

\end{document}